\documentclass[11pt]{article}
\usepackage{latexsym}
\usepackage{comment}
\usepackage{amssymb}
\usepackage{amsthm}
\usepackage{graphicx}
\usepackage{float}
\usepackage{subcaption}  
\newsavebox{\paretobox}

\usepackage{caption}
\usepackage{color}
\usepackage{hyperref}
\usepackage{amsmath}
\usepackage[normalem]{ulem}
\usepackage{cancel}   
\newtheorem{Theorem}{Theorem}[section]
\newtheorem{Proposition}{Proposition}[section]
\newtheorem{Lemma}{Lemma}[section]
\newtheorem{Corollary}{Corollary}[section]

\newtheorem{Assumption}{Assumption}[section]
\newtheorem{Remark}{Remark}[section]

\def \Sup{\displaystyle\sup}

\def \R{\mathbb{R}}
\def \E{\mathbb{E}}
\def \F{\mathbb{F}}
\def \P{\mathbb{P}}
\def \Ac{{\cal A}}
\def \Cc{{\cal C}}
\def \Dc{{\cal D}}
\def \Fc{{\cal F}}
\def \Hc{{\cal H}}
\def \Kc{{\cal K}}
\def \Lc{{\cal L}}

\def \Sup{\displaystyle\sup}

\def \R{\mathbb{R}}

\def \E{\mathbb{E}}
\def \F{\mathbb{F}}
\def \P{\mathbb{P}}

\def\1{\mbox{1\hspace{-0.25em}l}}

\def\beqs{\begin{eqnarray*}}
\def\enqs{\end{eqnarray*}}
\def\beq{\begin{eqnarray}}
\def\enq{\end{eqnarray}}

\title{Tractable bank capital structure: optimal control under Basel III constraints}

\begin{document}

\author{Erhan Bayraktar\footnote{Department of Mathematics, University of Michigan, USA, \sf erhan@umich.edu. Supported in part by the National Science Foundation and the Susan M. Smith Chair.}
\and  Etienne Chevalier\footnote{LaMME, Universit\'e d'\'Evry,
France, \sf etienne.chevalier@univ-evry.fr.} \and Vathana Ly
Vath\footnote{LaMME, ENSIIE, France, \sf lyvath@ensiie.fr.}\and Yuqiong Wang\footnote{Department of Mathematics, University of Michigan, USA, \sf yuqw@umich.edu.} }
\maketitle
\begin{abstract}
Banks must optimize risky investments, dividend payouts, and capital structure under tight Basel III solvency and liquidity constraints, while costly equity issuance serves as a distress-recovery tool. We formulate this as a stochastic control problem that reduces the high-dimensional balance-sheet dynamics to a tractable one-dimensional process in the asset-to-deposit ratio, with state-dependent investment limits. The resulting policy is simple and interpretable: pay dividends at an upper reflection barrier and, when needed, recapitalize only at the distress boundary, jumping to an optimal target level. We characterize these thresholds analytically and show their sensitivity to regulatory parameters. From a regulatory viewpoint, we use Monte Carlo simulation to solve an outer optimization problem and map the efficient frontier between shareholder value and survival probability, both with and without a leverage cap. In the illustrative parameter ranges studied here, tightening solvency requirements often yields the best safety--profitability trade-off.
\end{abstract}
\noindent {\bf Keywords~}: bank capital optimization, Basel III, optimal dividends, impulse and singular control, regulatory constraints, asset-to-deposit ratio, profitability--safety frontier.

\section{Introduction}

Banks continuously balance three competing objectives: paying dividends to shareholders, investing in risky assets to generate returns, and maintaining sufficient capital and liquidity buffers to satisfy regulatory requirements. Dividends are attractive to shareholders but reduce equity; risky investments can raise profitability but also increase the probability of distress; recapitalization through new equity issuance can restore solvency, but it is typically costly because of flotation costs, dilution, and market frictions. These trade-offs are amplified by regulatory requirements, such as Basel III's Tier 1 capital ratio and Liquidity Coverage Ratio (LCR), which impose hard constraints on leverage and liquid-asset holdings.

In this paper, we develop a tractable continuous-time model of a bank's optimal capital structure, dividend policy, and investment strategy under these two regulatory constraints. The bank collects deposits (paying interest at rate $r_L$) and invests in risk-free and risky assets financed by deposits and shareholders' equity. The manager controls the fraction $\pi_t$ invested in the risky asset, cumulative dividends, and the timing and size of equity issuance, which carries proportional costs $\kappa$ and $\kappa'$. Solvency and LCR requirements translate into a state-dependent upper bound $\overline{\pi}(y)$ on risky exposure, where $y = X/L$ is the asset-to-deposit ratio (total assets divided by deposits). The bank's objective is to maximize shareholder value---expected discounted dividends net of issuance costs. These regulatory constraints actively shape the bank's feasible strategies. As capital or liquidity buffers deteriorate, the bank may be forced to reduce risk taking, retain earnings, recapitalize at a cost, or liquidate. Issuance costs create a tension between immediate costs and the value of avoiding distress.

The mechanism studied here is complementary to, but distinct from, classical bank-run models. In the Diamond--Dybvig tradition, fragility arises because short-term liabilities finance illiquid long-term assets and depositors may run in a self-fulfilling equilibrium. We do not model strategic withdrawals or coordination among depositors. Instead, the runoff parameter in the LCR constraint is a reduced-form regulatory stress input, and distress is driven by equity depletion together with the investment restrictions imposed by solvency and liquidity regulation. This distinction is important for interpreting the results: the model is designed to isolate how ratio-based regulation changes a single bank's private payout, investment, and recapitalization incentives, not to provide a full theory of liquidity runs.

In terms of the asset-to-deposit ratio $y$, $y=1$ corresponds to the exhaustion of shareholders' equity. Rather than using the zero-equity level $y=1$ as the intervention boundary, we introduce an internal distress threshold $\underline y>1$, at which the bank chooses either to recapitalize or to liquidate. Introducing this strictly positive capital buffer is important for both the economic interpretation and the mathematical structure of the model. First, with purely proportional issuance costs, intervention at $Y=1$ would remove any positive intervention wedge associated with recapitalization and bring the lower-boundary control toward the proportional-cost singular capital-injection benchmark studied in the dividend and equity-issuance literature; see, for example, \cite{LokkaZervos2008}. A sequence of impulse interventions could therefore degenerate into a singular recapitalization policy. By taking $\underline y>1$, every recapitalization at the intervention boundary has a strictly positive minimum size, preserving the genuine impulse-control nature of equity issuance. Second, at $y=1$, the regulatory constraint forces risky investment to zero and the diffusion coefficient of the ratio process vanishes. The local drift then reduces to $r-r_L$. Consequently, when $r>r_L$, the zero-equity boundary is locally repelling. Introducing $\underline y>1$ provides an operative intervention boundary without requiring a standing restriction on the relative magnitudes of $r$ and $r_L$.

The threshold also has a practical going-concern interpretation. Banks may raise equity while they remain solvent to preserve a capital cushion rather than waiting until shareholders’ equity is exhausted. For example, in discussing JPMorgan’s acquisition of Washington Mutual during the financial crisis, Jamie Dimon described raising an additional $\$11$ billion of equity that he did not regard as immediately necessary, in order to maintain the strength of the bank’s balance sheet if conditions deteriorated further \cite{acquiredDimon2025}. Dimon has also argued that a troubled bank should be required to raise capital substantially before failure \cite{dimon2017}. More recently, he proposed a sequence of capital thresholds under which restrictions on distributions and dividend cuts precede mandatory capital raising \cite{dimon2026}. These observations support interpreting $\underline y>1$ as a precautionary, going-concern recapitalization threshold distinct from the zero-equity bankruptcy boundary.

Finally, $\underline y$ plays a different role from the regulatory parameters $a_1$, $a_2$, and $a_3$. The latter continuously restrict the bank's admissible risky investment, whereas $\underline y$ determines when recapitalization or liquidation occurs. This distinction motivates the delegation analysis in Section~\ref{sec_delegation}, where the regulator sets the external requirements while the bank is allowed to choose its own intervention threshold. It also motivates further study of regulator--bank interactions of the Stackelberg type.

Dynamic models of bank capital structure with costly recapitalization form a growing literature. Delayed equity issuance is studied in \cite{Peukep06}. The effect of capital regulation on banks' portfolio choices is studied in \cite{rochet}, while \cite{HM} and \cite{bolton} examine how liquidity and leverage rules jointly shape capital structure, default risk, and refinancing. In particular, \cite{bolton} studies dynamic bank balance-sheet management under leverage regulation, and its Internet Appendix \cite{boltonIA} considers a liquidity requirement motivated by the LCR. Liquidity is often treated as an exogenous buffer \cite{diamond} or through runoff and haircut assumptions \cite{bala,cetina,doer}; \cite{doer} argues that although the LCR reduces reliance on short-term funding, its optimal design remains an open question. Early work on singular dividend control includes \cite{chotakzho03} and \cite{jeashi95}, while combined singular-impulse and switching problems are studied in \cite{soto11} and related work. For a broader treatment of continuous-time models in corporate finance, banking, and insurance, see \cite{morenoRochet2018}.

From a mathematical perspective, the problem is a combined singular-impulse control problem for diffusions: singular control for dividends (reflection barriers, \cite{harrison1}; \cite{jeashi95}) and impulse control for recapitalization (jump structure, \cite{harrison2,orm}; \cite{Alvarez2008}). L{\o}kka and Zervos \cite{LokkaZervos2008} study optimal dividends and equity issuance under proportional issuance costs. A related corporate finance model is studied in \cite{decamps2}. More broadly, \cite{bcw2011} show that, with costly external financing, optimal payout and external financing policies are governed by endogenous boundary policies; extensions incorporating investment/growth appear in \cite{decvil05}, \cite{lyphvi07}, \cite{guotom08}, and others. We build on this threshold-policy logic rather than claiming that it is new per se. Our contribution is to prove that this structure persists in a regulated banking model in which solvency and liquidity requirements jointly create a state-dependent, kinked investment constraint, and to use the resulting tractability for regulatory comparative statics.

First, we show how Tier 1 solvency and LCR requirements with haircuts and runoff can be represented as a single state-dependent upper bound $\overline{\pi}(y)$ on risky investment. This formulation turns the regulatory environment into an explicit constraint on the bank's investment opportunity set. Exploiting homogeneity, the two-dimensional $(X,L)$ problem reduces to a one-dimensional singular-impulse problem in the asset-to-deposit ratio $Y_t=X_t/L_t$, making the regulated control problem analytically tractable. The state space is $[\underline y,\infty)$ with $\underline y>1$. This choice also avoids degeneration into an infinitesimal singular control problem and avoids the difficulties associated with a locally repelling boundary when $r>r_L$.

Second, while boundary policies for payout and external financing are familiar from dynamic corporate finance, it is not immediate that this structure persists when the bank's investment set is state-dependent and kinked by regulation. We establish the viscosity characterization, prove concavity of the value function, and derive the resulting dividend and recapitalization thresholds. Theorem \ref{optimal_strategy} is therefore best interpreted as a robustness and extension result: the classical threshold structure survives in a Basel-style environment with dual state-dependent regulatory constraints.

Third, we use the analytical characterization to study regulatory design. We explicitly characterize $y^*$ and the post-issuance target, trace sensitivities to regulatory parameters $a_1$ (solvency), $a_2$ (runoff), and $a_3$ (haircut), and quantify the value of the recapitalization option. We then formulate and solve a regulator's problem---maximizing bank value subject to a survival-probability constraint over a finite horizon---with and without a no-leverage restriction ($\pi \leq 1$). Monte Carlo evaluation of the resulting Pareto frontier shows that $a_1$ dominates in both regimes, suggesting that tightening capital requirements is often the most efficient safety tool within the parameter ranges studied; see Section \ref{regulator_optimization}. We then allow the bank to choose its own intervention threshold after the external regulatory requirements have been fixed. The bank generally prefers a higher trigger than a coordinated benchmark designed to limit the number of recapitalizations. The associated shareholder-value gain from this delegation is small, but the change in recapitalization behavior can be large: under fixed regulation, a higher trigger generates more frequent, smaller capital injections and greater total issuance. Tighter solvency regulation can reduce both intervention frequency and total issuance even after the bank optimally chooses its trigger. Finally, we show that this optimization depends strongly on recapitalization costs: when external equity is inexpensive, the bank prefers higher intervention thresholds, whereas higher issuance costs $\kappa'$ push the preferred trigger downward and sharply reduce reliance on capital issuance.

Managerially, the analysis turns the continuous-time control solution into operating triggers that can be interpreted by bank managers and regulators: a capitalization level above which excess funds are paid out, a distress boundary at which recapitalization becomes valuable, a post-issuance target, and a regulatory switching point at which the binding constraint changes from solvency to liquidity or vice versa. The economic content of the model is therefore not that threshold policies are new, but that Basel-style constraints move these thresholds by reshaping the bank's feasible risk--return trade-off at each capitalization level. In particular, the model predicts that tightening solvency regulation tends to raise the retained-capital buffer before dividends are paid, while liquidity regulation matters most in the region where the LCR component of $\overline{\pi}(y)$ binds. The no-leverage case further clarifies when liquidity rules may become locally nonbinding because an additional market or balance-sheet restriction already limits risky investment. The delegation analysis adds a separate operational implication. The solvency requirement remains the main regulatory instrument, but the intervention threshold determines how aggressively the bank relies on recapitalization. When external equity is inexpensive, the bank prefers a higher intervention threshold and recapitalizes more frequently. Consequently, delegating intervention thresholds can matter even though its effect on shareholder value is small. The motivation for restricting the intervention threshold is stronger when the regulator places direct weight on the number of issuances and the bank faces low recapitalization costs.

The regulatory exercise should be interpreted as a single-bank constrained-efficiency calculation rather than a full welfare analysis. Because the model has one representative bank, it does not internalize systemic externalities such as correlated exposure to regulatory boundaries, fire-sale spillovers, deposit-insurance costs, or the possibility that many banks simultaneously reduce dividends, sell risky assets, or attempt to recapitalize. These aggregate channels are important in practice and remain outside the scope of the present model.

The paper proceeds as follows. Section \ref{sec2} describes the balance-sheet dynamics, constraints, and one-dimensional reduction. In Section \ref{sec3}, we first prove the viscosity characterization and concavity of the value function, which we then use to derive the explicit thresholds for issuance and dividends. Section \ref{sec4} presents numerical results on value and policy sensitivities, the regulator's safety--profitability frontier, and the interaction between solvency regulation, delegated intervention thresholds, and recapitalization costs. Proofs appear in Appendix \ref{appendix}.

\section{Balance-sheet dynamics, regulatory constraints, and the bank's optimization problem}\label{sec2}
Let $(\Omega, \F, \P)$ be a probability space equipped with a
filtration $\F= (\Fc_t)_{t \geq 0}$ satisfying the usual
conditions. All random variables and stochastic
processes are defined on this probability space.
Let $W$ and $B$ be two correlated $\F$-Brownian motions with
correlation coefficient $c$ where $|c|<1$; that is, $d[W,B]_t=c\,dt$ for all $t$. We consider a bank whose liabilities consist of customer deposits, denoted by $L_t$ at time $t$. We
model the process $L$ using the following
stochastic differential equation:
\beq dL_t = L_t
\left(\mu_{_L} dt+\sigma_{_L}dW_t\right), \: L_{0}  =  \ell, \enq
where $\sigma_{_L}$ is a positive constant and
$\mu_{_L}:=\gamma+r_{_L}$. Here, $\gamma\in\mathbb{R}$ is the
exogenous growth rate of deposits and $r_{_L}\geq 0$ is the interest
rate paid by the bank to its clients. The bank may invest in a
risk-free asset with a constant interest rate $r>0$ or in a risky
asset whose value process $S$ satisfies the following
stochastic differential equation:
\beq dS_t=S_t\left(\mu dt+\sigma dB_t\right),\: S_0 = s, \enq
where $\mu\in\mathbb{R}$, $\sigma>0$.
The diffusion specification is a deliberately parsimonious continuous-time approximation. Brownian shocks represent frequent aggregate fluctuations in asset returns and deposit growth and make it possible to isolate the effect of the regulatory constraints on the bank's controls. The specification does not capture jumps, fire sales, abrupt market-access failures, or strategic withdrawals; incorporating those features would require additional state variables or jump terms and is outside the scope of the present analysis.
Let $X_t$ denote the bank's total assets at time $t$, and let $\pi_t$ denote the fraction invested in
the risky asset. Accordingly, $(1-\pi_t)X_t$ is invested in the risk-free asset, and $\pi_t X_t$ is invested in the risky asset at time
$t$. By the balance-sheet identity, we have
$$X_t=F_t+L_t\quad\forall t\geq0,$$
where $F_t$ denotes shareholders' equity at time $t$. The bank's manager controls both the allocation of assets between the risk-free and risky assets and the bank's capital
through equity issuance and dividend payments. We
consider a control strategy
\[
{\widehat{\alpha}} =
((\tau_n)_{n\in\mathbb{N}^*},(\widehat{\xi}_n)_{n\in\mathbb{N}^*},\widehat{Z},\pi),
\]
where the $\F$-adapted c\'adl\'ag nondecreasing process
$\widehat{Z}$ represents the total amount of dividends distributed,
with $\widehat{Z}_{0^-} = 0$. The strictly increasing sequence of
stopping times $(\tau_n)_{n\in\mathbb{N}^*}$ represents the times at which
the manager decides to issue new capital. The $\Fc_{\tau_n^-}$-measurable random variable $\widehat{\xi}_n\in(0,+\infty)$ represents
the amount of capital issued at $\tau_n$. The process $\pi$ represents the
fraction of the bank's assets invested in the risky asset. The equity process associated with a control ${\widehat{\alpha}}$ then evolves as follows:
$$\left\lbrace
\begin{array}{lllr}
dF_t & = & -r_{_L}
L_tdt-d\widehat{Z}_t+(1-\pi_t)X_trdt+\pi_tX_t\left(\mu dt+\sigma
dB_t\right) & \textrm{ for }\tau_i<t<\tau_{i+1}\\
F_{\tau_i} & = &
(1-\kappa)F_{\tau^-_i}+(1-\kappa^\prime)\widehat{\xi}_i &
\end{array} \right.$$
Here, $\kappa^\prime,\kappa>0$ are issuance-cost parameters. When
issuing capital at time $t$, the manager pays a cost
proportional to the amount issued, while $\kappa F_{t^-}$ is the
cost of compensating existing shareholders for dilution. We also assume
that $\widehat{\xi}_i$ is large enough to ensure that
$F_{\tau_i}>F_{\tau_i^-}$, i.e.,
$\widehat{\xi}_i>\frac{\kappa}{1-\kappa^\prime}F_{\tau^-_i}$.
Otherwise, the manager would be better off avoiding issuance and distributing dividends instead. The corresponding wealth process $X$ then solves
\beq\left\lbrace
\begin{array}{lrl}
dX_t & = & \left((1-\pi_t)r X_t+\pi_t\mu X_t +\gamma
L_t\right)dt+\pi_t\sigma X_tdB_t+\sigma_L L_tdW_t-d\widehat{Z}_t,\\
&& 
\textrm{ for }\tau_i<t<\tau_{i+1}, \\
X_{\tau_i} & = &
X_{\tau^-_i}+(1-\kappa^\prime)\widehat{\xi}_i-\kappa F_{\tau^-_i}
\end{array}
\right.\enq
The accounting insolvency level corresponds to $X_t=L_t$, or equivalently $Y_t=1$, at which shareholders' equity is exhausted. We do not use this level as the operative intervention boundary. Instead, we fix a supervisory distress threshold $\underline y>1$. Whenever the pre-intervention state reaches $\underline y$, the bank must either recapitalize immediately or liquidate. If recapitalization is chosen, the capital injection is executed first and the bank continues from the resulting post-issuance state above $\underline y$. Otherwise, liquidation occurs and shareholders receive the residual value $X_t-L_t=L_t(\underline y-1)$. Accordingly, we define the liquidation time by $T^{\widehat\alpha}:=\inf\left\{t\geq0:X_{t^-}\leq \underline y L_t\ \text{and}\  t\notin\{\tau_n\}_{n\geq 1}\right\}$. Allowing recapitalization at this boundary assumes that supervisory intervention occurs while an equity issue remains feasible, for example through a prearranged rights issue or a regulator-facilitated recapitalization. This is a modeling assumption rather than a claim that distressed banks always retain market access. Fixed issuance costs, endogenous access to equity markets, or adverse-selection costs that increase near distress could instead make liquidation optimal or induce earlier issuance.
We fix a constant shareholder discount rate $\rho>0$. The shareholders receive cumulative dividends $\widehat{Z}$ until bankruptcy. At an equity issuance time $\tau_n$, a total amount of $(1-\kappa')\widehat\xi_n$ is added to equity, while the dilution cost $\kappa F_{\tau_n^-}$ is deducted at issuance. Accordingly, the net issuance cash flow to shareholders at $\tau_n$ is $\widehat{\xi}_n-\kappa
F_{\tau_n^-}$. Given initial liability $\ell>0$ and initial wealth
$x>0$, the shareholders' present value under the policy
${\widehat{\alpha}}$ is defined by
$$J^{{\widehat{\alpha}}}(\ell,x)=\mathbb{E}_{\ell,x}\Big[\int_0^{T^{\widehat{\alpha}}}e^{-\rho
t}d\widehat{Z}_t-\sum_{n\geq 1}e^{-\rho\tau_n}(\widehat{\xi}_n-\kappa
F_{\tau_n^-})\1_{\lbrace \tau_n\le
T^{\widehat{\alpha}}\rbrace} + e^{-\rho T^{\widehat \alpha}}(X_{T^{\widehat \alpha}}-L_{T^{\widehat \alpha}})\Big].$$
We now introduce the regulatory constraints that reflect the institutional features of banking. The
first is a solvency requirement that
measures the bank's ability to absorb losses without default. In our setting, the solvency ratio is the ratio of shareholders' equity to risky asset holdings:
$\frac{F_t}{\pi_t X_t}\geq a_1 $
for some $a_1\in(0,1)$.
Equivalently, by introducing the asset-to-deposit ratio $Y_t:=\frac{X_t}{L_t}$, this condition can be rewritten as
$$1-\frac{1}{Y_t}\geq a_1\pi_t.$$
The second constraint is the Liquidity Coverage Ratio (LCR), defined as the
ratio of high-quality liquid assets (HQLA) to cash outflows
over 30 days. We focus on the potential runoff of retail deposits and model the 30-day net outflow as a fixed fraction $a_2\in(0,1)$ of liabilities, i.e., $a_2L_t$. We assume that risk-free assets qualify fully as HQLA. In addition, we allow risky assets to contribute to HQLA only after applying the regulatory haircut prescribed by the Basel III LCR framework \cite{baco13}. In other words, a fraction $a_3\in (0,1)$ is excluded from HQLA. Hence, when the bank invests a fraction $\pi_t$ in risky assets, the HQLA level is
$(1-\pi_t )X_t+(1-a_3)\pi_t X_t = X_t -a_3 \pi_t X_t$. 
The liquidity constraint can therefore be expressed as $\frac{X_t - a_3\pi_t X_t }{a_2 L_t} \geq 1\textrm{ for all }t\geq
0,$
or in terms of the asset-to-deposit ratio $Y$,
$$1 - \frac{a_2}{Y_t} \geq a_3\pi_t.$$
This motivates the function $\overline{\pi}$, defined on
$\lbrack 1,+\infty)$ by
\begin{equation}
\label{barpi}
    \overline{\pi}(y)=\min\left(\frac{1}{a_1}(1-\frac{1}{y}),\ \frac{1}{a_3}(1-\frac{a_2}{y})\right).
\end{equation}
Economically, $\overline{\pi}(y)$ links regulation directly to the bank's investment opportunity set. When the bank is close to the distress boundary $y=\underline y$, the solvency constraint forces risky investment to be small, so risk taking is restricted precisely when the equity buffer is thin. As $y$ increases, the bank can take more risk, but the binding regulatory constraint may switch between the solvency and liquidity requirements. This kinked investment cap is the main channel through which Basel-style regulation affects the bank's payout and recapitalization decisions. We distinguish the accounting constraints from the intervention trigger at which normal operations end. The regulatory parameters $a_1,a_2,a_3$ determine the admissible risky-investment set through $\overline{\pi}(y)$. Separately, the distress trigger $\underline{y}$ marks the point at which the bank must either recapitalize or liquidate, before its equity is fully exhausted.

The function $\overline \pi $ gives the regulatory upper bound on the proportion of assets invested in the risky asset. We require the intervention times to satisfy $\tau_n<\tau_{n+1}$ and $\tau_n\uparrow\infty$. The set of admissible controls, denoted by $\widehat{\Ac}$, is defined by{\small$$\widehat{\mathcal{A}}=\{{\widehat{\alpha}} =
((\tau_n)_{n\in\mathbb{N}^*},(\widehat{\xi}_n)_{n\in\mathbb{N}^*},\widehat{Z},
\pi): \forall 0 \leq t \leq T^{\widehat{\alpha}}, 0\leq
\pi_t\leq\overline\pi(\frac{X_t}{L_t}),\ \forall n\geq 1:
\widehat{\xi}_n>\frac{\kappa}{1-\kappa^\prime}F_{\tau_n^-}\}.$$}
We call a control admissible if in addition, $\pi$ is progressively measurable, $\widehat Z$ is adapted, nondecreasing and c\`adl\`ag with $\widehat Z_{0^-}=0$, and dividend jumps before liquidation leave the next state in the state space. We ask the issuance sizes are positive and $\mathcal{F}_{\tau_n^-}$-measurable. The control satisfy $0\leq \pi_t\leq \overline \pi(Y_t)$ prior to liquidation. Moreover, away from issuance times, the dividend jumps satisfy $\Delta \widehat Z_t\leq X_{t^-}-\underline yL_t$, so the dividends do not move the state below $\underline y$. We further assume all controls are stopped at liquidation, and the discounted dividend and issuance cash flows to be integrable, and the localized payoff to be uniformly integrable. Our value function is defined by \beq
\widehat{v}(\ell,x)=\sup_{{\widehat{\alpha}}\in\widehat{\mathcal{A}}}J^{\widehat{\alpha}}(\ell,x)\quad\textrm{for
} (\ell,x)\in\mathcal S :=\{(\ell,x)\in\lbrack 0,+\infty)^2: x\geq\underline y
{\ell}\}.\enq
We now state a result that transforms our initial
two-dimensional problem into a one-dimensional control problem in terms of the asset-to-deposit ratio $Y$, which is also a natural state variable from a regulatory perspective.
\begin{Proposition}[Reduction to the asset-to-deposit ratio]
\label{thm1}
Let
\[
\alpha:=((\tau_n)_{n\in\mathbb{N}^*},({\xi}_n)_{n\in\mathbb{N}^*},{Z},\pi),
\]
where $(\tau_n)_{n\in\mathbb{N}^*}$ is an increasing sequence of
stopping times, $({\xi}_n)_{n\in\mathbb{N}^*}$
a sequence of positive $\mathcal{F}_{\tau_n^-}$-measurable random
variables, and $Z$ an increasing process. Let the process
$Y^\alpha$ solve the following stochastic differential
equation:
$$\begin{cases}
    dY^{\alpha}_t& =  \left( Y^\alpha_t\left[\mu(\pi_t)-\mu_L\right]+\gamma\right)dt+\pi_tY^{\alpha}_t\sigma dB_t+\sigma_L(1-Y^{\alpha}_t)dW_t-d{Z}_t,\quad \tau_n<t<\tau_{n+1}\\
Y_{\tau_n}  &=  (1-\kappa)Y_{\tau_n^-}+(1-\kappa^\prime)\xi_n+\kappa,
\end{cases}$$
where $\mu(\pi)=(1-\pi)r+\pi\mu $. Define the stopping time $T^\alpha=\inf\{t\ge 0:\ Y^\alpha_t\leq \underline y, t\notin \{\tau_n\}_{n\geq 1}\}$. Then,
$\widehat{v}(\ell,x)=\ell v(\frac{x}{\ell})$, for all $\ell>0$ and $x\geq \underline y \ell$, where $v: [\underline y,\infty)\to \R$ is given by
\[v(y)=\sup_{{{\alpha}}\in\mathcal{A}}\mathbb{E}_{y}\Big[\int_0^{T^{{\alpha}}}e^{-\rho_L
t}d{Z}_t-\sum_{n\geq 1}e^{-\rho_L\tau_n}({\xi}_n-\kappa(Y^\alpha_{\tau_n^-}-1))\1_{\lbrace
\tau_n\le T^{\alpha}\rbrace}+ e^{-\rho_L T^{\alpha}} (\underline y-1)\Big]\]
with $\rho_L:=\rho-\mu_L$.
\end{Proposition}
The equivalent equity-to-deposit ratio $\widetilde{Y}_t:=F_t/L_t$ could also be used as the state variable, since $\widetilde{Y}_t=Y_t-1$. We work with $Y_t=X_t/L_t$ because the zero-equity boundary is then $Y_t=1$ and the regulatory bounds in \eqref{barpi} take a compact form. The reduction relies on the homogeneity of the balance sheet and the objective. This is a deliberate modeling restriction. Fixed issuance costs, nonlinear issuance costs, taxes, deposit insurance premia, endogenous deposit rates, balance-sheet size effects, multiple asset classes, or regulatory constraints depending on dollar levels rather than ratios would generally break the one-dimensional reduction. The advantage of the present specification is that it preserves the ratio-based nature of Basel-style solvency and liquidity rules while keeping the singular-impulse control problem analytically tractable.
\section{Characterization of the optimal strategies}\label{sec3}
In this section, we study the value function, establish its analytical properties, and then use them to characterize the optimal strategies.
\subsection{Analytical properties of the value function}
This subsection establishes the analytical properties of the value function, which underpin the structure of the optimal strategies. Its main result is that $v$ is the unique viscosity solution of the following variational inequality:
\begin{equation}
    \label{HJB}
    \begin{cases}
        0 =  \min \{ \rho_L v(y) - \Sup_{0\leq\pi\leq \overline{\pi}(y)}\Lc^\pi v(y); v^\prime(y)- 1; v(y)-\mathcal{H}v(y) \},\\
v(\underline y)  =  \max\left(\underline y-1,\ \mathcal{H}v(\underline y)\right) . 
    \end{cases}  
\end{equation}
Here, the impulse operator $\mathcal{H}$ is defined by
\beqs
\mathcal{H} \varphi (y)& = & \sup_{\xi>\frac{\kappa}{1-\kappa^\prime}(y-1)}\Big[ \varphi((1-\kappa)y+(1-\kappa^\prime)\xi+\kappa)-\xi+\kappa(y-1)\Big],
 \enqs
and the operator $\Lc^\pi$ is defined by 
\beqs \Lc^\pi \varphi (y)&=& \frac{1}{2}\left(\pi^2\sigma^2y^2+2\pi c\sigma\sigma_Ly(1-y)+\sigma_L^2(1-y)^2\right)\varphi^{\prime\prime} +\left( y\left[\mu(\pi)-\mu_L\right]+\gamma\right)\varphi^\prime.\enqs
Here, $\mu(\pi)=(1-\pi)r+\pi\mu$. For convenience, we write the impulse operator in post-issuance coordinates by setting $z= (1-\kappa)y +(1-\kappa')\xi +\kappa$. Then $z>y$ and
\beqs
\mathcal{H} \varphi (y)& = & \sup_{z>y}\Big[ \varphi(z)-\frac{z-y}{1-\kappa'}-\frac{\kappa\kappa'}{1-\kappa'}(y-1)\Big].
 \enqs
We first observe that, for any bounded stopping time $\theta$, the value function $v$ satisfies the following dynamic programming principle (DPP):
\begin{equation}\label{dpp}    
\begin{aligned}
v(y)=&\sup_{{{\alpha}}\in\mathcal{A}}\mathbb{E}_{y}\Big[\int_0^{\theta\wedge T^{\alpha}}e^{-\rho_L
t}d{Z}_t-\sum_{\tau_n\leq \theta\wedge T^{\alpha}}e^{-\rho_L\tau_n}({\xi}_n-\kappa(Y^\alpha_{\tau_n^-}-1))\\
&\qquad\qquad+e^{-\rho_L \theta }v(Y^{\alpha}_{\theta})\1_{\theta <T^{\alpha}}+e^{-\rho_LT^\alpha}(\underline y-1)
\mathbf 1_{\{\theta\geq T^\alpha\}}.\Big].
\end{aligned}
\end{equation}
\subsubsection{Lower and upper bounds for the value function}
We introduce notation used throughout the analysis. Recall that the bank allocates a fraction $\pi_t$ of its assets to the risky asset, and the instantaneous expected return is $\mu(\pi_t) = r+(\mu-r)\pi_t$. Since $\mu$ is affine in $\pi_t$ and the regulatory constraints impose $0\leq \pi_t\leq \overline \pi(y)$, $\mu(\pi)$ is maximized at an endpoint. We define the drift-maximizing strategy by
$\pi^*(y)=\overline{\pi}(y)\1_{\{\mu\geq r\}}$,
and the corresponding maximized instantaneous expected return by 
$\mu^*(y):=r+\pi^*(y)(\mu-r)=\max\{ (1-\pi)r+\pi\mu: 0\leq \pi\leq \overline{\pi}(y)\}$.
We emphasize that $\pi^*$ is not necessarily optimal for the full control problem, since optimality also depends on the correlation and volatility parameters, as well as the issuance and dividend strategies through the variational inequality. We refer to $\pi^*$ as the myopic strategy throughout the rest of the paper.
Recall that the upper bound $\overline{\pi}(y)=\min\left(\frac{1}{a_1}(1-\frac{1}{y});\ \frac{1}{a_3}(1-\frac{a_2}{y})\right)$ increases in $y$, and it converges as $y\to \infty$ to
$\lim_{y\to \infty } \overline{\pi}(y)= \frac{1}{\max(a_1, a_3)} =: \frac{1}{\bar a}$. It follows that 
\begin{align*}
    \sup_{y\geq 1} \mu^*(y ) = \begin{cases}
        r, \quad \mu\leq r,\\
        \frac{\mu +(\bar a-1) r}{\bar a}, \quad \mu> r.
    \end{cases}
\end{align*}
 We now impose a parameter restriction under which discounting dominates the maximal growth permitted by regulation, thereby ensuring that the problem is well posed.
\begin{Assumption}\label{assump1}
Throughout the rest of the paper, we assume that $\rho> \max(\mu_L,r+\frac{(\mu-r)^+}{\bar a})$, where $\bar a = \max(a_1,a_3)$.    
\end{Assumption}
\begin{Proposition}\label{lower_upperbound}
Define $M: = \sup_{z\geq \underline y} \{(\mu^*(z)-\rho)z\}$ and let $K: = \frac{1}{\rho_L}\max\left(-\rho_L,M+\gamma\right)$. Under Assumption \ref{assump1}, $M<\infty$. Then, for all $y\geq \underline y$, the value function satisfies
\begin{equation}\label{bounds}
    y-1\leq v(y)\leq y +K.
\end{equation}  
\end{Proposition}
Indeed, for any initial state $y\geq \underline y$, the bank can pay dividends of size $y - \underline y$ and then liquidate at the boundary, receiving $\underline y-1$. In particular, if $-\rho_L\geq M+\gamma$, then $v(y)=y-1$, and the optimal policy is to pay $y-\underline y$ immediately and liquidate. Accordingly, throughout the rest of the paper, we also assume that the parameters satisfy $-\rho_L<M+\gamma$.
\begin{Remark}
At the hypothetical zero-equity level $y=1$, the solvency constraint forces $\pi=0$ and the diffusion coefficient of the ratio process vanishes. Since $\mu_L=\gamma+r_L$, its local drift at this point is $r-\mu_L+\gamma=r-r_L$. Thus, $y=1$ is locally repelling when $r>r_L$ and locally attracting when $r<r_L$. This observation does not impose a restriction on the present model, because normal operation terminates at the strictly higher intervention threshold $\underline y>1$. Accordingly, both $r>r_L$ and $r<r_L$ are admissible, subject to Assumption~\ref{assump1} and the other standing conditions.\end{Remark}

\subsubsection{Viscosity characterization of the value function}
We first state the continuity of $v$, which is needed for both the viscosity framework and the stable characterization of the free boundaries. We then show in Proposition \ref{viscosity} that the value function is uniquely characterized as the solution of the variational inequality.
\begin{Proposition}[Viscosity characterization of the value function]\label{viscosity}
The value function $v$ is the unique continuous function on
$\lbrack \underline y,+\infty)$ that satisfies the linear growth condition $y-1\leq v(y)\leq y+K$ and is a viscosity solution of \eqref{HJB}.
\end{Proposition}
The key structural property of the value function is concavity because it shapes the geometry of the optimal control regions. Economically, concavity of $v$ indicates the diminishing marginal value of capital: an additional unit of $y$ is most valuable near the distress boundary, which leads to a single dividend threshold.
\begin{Proposition}[Concavity of the value function]\label{concavity} $v$ is concave on $\lbrack \underline y,+\infty)$.
\end{Proposition}
Finally, we define the continuation, dividend, and issuance regions, and establish interior regularity of $v$, which yields a classical solution in the continuation region. This is not merely technical: it provides the smooth-fit conditions used to compute the dividend barrier, and it supports the numerical analysis.
\begin{Corollary}\label{regularity} The value function $v$ is $C^1$ on $\{y>\underline y: v(y)>\mathcal{H}v(y)\}$. Define the issuance region $\Kc$, dividend region $\Dc$, and the continuation region $\Cc$ by 
\beq \mathcal{K}
&:=& \left\{ y \geq \underline y~,  v(y) \; = \; \mathcal{H}v(y) \right\} \label{Capital issuing region}\\
\Dc &=& \overline{{\rm int}\left(\left\{ y \geq \underline y~, v'(y) = 1 \right\}\right)}, \label{Dividend region}  \\
\Cc &=&  \{y>\underline y:  v(y)>\mathcal{H}v(y), \enskip D^+v(y)>1\} . \enq 
The set $\Cc$ is open and $\Kc \cap \Dc = \emptyset$. Furthermore,
$v$ is $C^2$ on the open set $\Cc\cup {\rm int}(\Dc)$, and the HJB equation
$\rho_L v(y) - \Sup_{0\leq \pi\leq\overline{\pi}(y)}\Lc^\pi v(y) = 0$ holds in the classical sense for $y \in \Cc$.
\end{Corollary}

\subsection{Optimal dividend and capital issuance strategies}
\label{sec_optimal_strategies}
In this section, we translate the variational inequality into explicit characterizations of the optimal strategies. Boundary policies for payout and costly external financing are well known in dynamic corporate finance, notably in \cite{bcw2011}. The purpose of the analysis here is to show that this familiar structure continues to hold in our regulated banking setting, despite the state-dependent and kinked investment cap induced by the joint solvency and liquidity constraints. First, dividends are optimal when retaining one more unit of capital is no more valuable than paying it out. The concavity of the value function means that once paying dividends is optimal, it remains optimal for all larger $y$, which implies a single barrier structure. The same concavity argument also governs capital issuance. Theorem \ref{optimal_strategy} shows that the optimal strategy has a simple two-sided structure: dividends are paid at an upper barrier, while recapitalization, when optimal, is triggered at the distress boundary and jumps to a post-issuance target.
\begin{Theorem}\label{optimal_strategy}\textbf{Optimal dividend and capital issuance strategy.}
The optimal strategy is characterized as follows.
\begin{enumerate}
\item[(i)] The dividend region is of the form
$\Dc=\lbrack y^*,+\infty)$, where $y^*: = \inf\{y\geq \underline y: D^+v(y)=1\}$, and $y^*<\infty$. In addition, $y^*$ satisfies $\rho_L v(y^*)=\gamma-(\mu_L-\mu^*(y^*))y^*$.
\item[(ii)] The issuance region $\Kc\subset\{\underline y\}$. Moreover, if $\Kc\not=\emptyset$, then $D^+ v(\underline y)>\frac{1}{1-\kappa^\prime}$, and there exists a post-issuance target $y_{post}^*$ such that $D^-v(y_{post}^*)\geq \frac{1}{1-\kappa^\prime}\geq D^+v(y_{post}^*)$. The value function at $y=\underline y$ satisfies
 $v(\underline y)=v({y_{post}^*})-\frac{y_{post}^*-\underline y}{1-\kappa^\prime}-\frac{\kappa \kappa'}{1-\kappa'}\left(\underline y-1\right)$.
\end{enumerate}
\end{Theorem}
\begin{Remark}
In addition, if the parameters satisfy the following conditions, then $y^*$ is uniquely identified by $\rho_L v(y)=\gamma-(\mu_L-\mu^*(y))y$:
\begin{align*}
    \begin{cases}
        \rho \neq r+\frac{\mu-r}{a_1},\rho \neq r+\frac{\mu-r}{a_3},\quad \mu>r,\\
        \rho \neq r,\quad \mu\leq r.
    \end{cases}
\end{align*}
\end{Remark}
Economically, Theorem \ref{optimal_strategy} says that costly external finance makes the bank reluctant to recapitalize preemptively: issuance is used only at the distress boundary, and then only to jump to the level at which the marginal value of internal funds equals the marginal cost of new equity. Conversely, when capitalization is high, the marginal value of retaining one more unit falls to one, and the excess is paid out as dividends. The regulatory cap $\overline{\pi}(y)$ affects where these thresholds lie by changing both the expected growth and the risk of the asset-to-deposit ratio.
The decomposition of the mechanism is as follows. The one-sided issuance region is driven primarily by the distress boundary, proportional issuance costs, and concavity of the value function. The regulatory constraints do not create the threshold structure from scratch; rather, they alter the continuation value by changing the admissible risk-return trade-off at each capitalization level. This is why the theorem is best read as a robustness result for classical payout/financing thresholds under a kinked Basel-style investment set. If one introduced fixed issuance costs, endogenous market access, or adverse-selection costs that rise near distress, earlier recapitalization or inaction at the boundary could become optimal.
\section{Sensitivity analysis and safety-profitability frontiers under regulation}
\label{sec4}
\subsection{Value function and optimal strategies}
\footnote{The numerical results in this section were generated using scripts available on GitHub at \url{https://github.com/yuqiongwang/bank_capital_structure}.} We begin by plotting the value function $v$ and comparing it with the benchmark case in which equity issuance is not allowed, denoted by $v^{HJB}$. To isolate the value of recapitalization, the
benchmark uses the same $\underline y$ and the same investment constraints as the main model, but does not allow equity issuance. The function $v^{HJB}$ solves
\[ 0 =\min \{ \rho_L v^{HJB}(y) - \Sup_{0\leq\pi\leq \overline{\pi}(y)}\Lc^\pi v^{HJB}(y); (v^{HJB})^\prime(y)- 1\},\quad v^{HJB}(\underline y)=\underline y-1.\]
The no-issuance benchmark is computed on the same state domain as the full case, with common boundary $v^{HJB}(\underline y)=\underline y-1$, so the comparison differs only by the availability of recapitalization. Throughout the numerical analysis, unless otherwise stated, we use the baseline parameter set $r=0.02, \mu=0.04, \mu_L=0.03, \rho=0.12, \gamma=0.02$, together with $\sigma=0.08, \sigma_L=0.03,c=0.20, \kappa=0.01, \kappa'=0.02$. The baseline regulatory parameters are $(a_1,a_2,a_3)=(0.045,0.05,0.30)$, the intervention threshold is $\underline y=1.02$, and value functions are evaluated at $y_0=1.20$. Here, $r$ and $\mu$ are the risk-free and risky-asset returns, $\mu_L$ is the drift of deposits, $\rho$ is the shareholders' discount rate, $\sigma$ and $\sigma_L$ are asset and deposit volatilities, $c$ is their shock correlation, and $\gamma$ is the exogenous component of deposit growth. These values constitute an illustrative parameterization rather than an empirical calibration: they satisfy the model's well-posedness conditions and are chosen to make the regulatory mechanisms transparent. Accordingly, the numerical conclusions are stated only for the parameter ranges examined below.

We solve the VI and the no-issuance HJB on a bounded state grid using a monotone finite-difference scheme. The maximization over risky investment is performed on a discrete control grid, and the dividend-gradient and issuance-intervention conditions are imposed at each iteration. The baseline calculations use 301 state points over the displayed computational range, 151 control points, a pseudo-time step of $0.05$, and a convergence tolerance of $10^{-6}$; the upper boundary uses the asymptotic unit-slope condition. The illustrative trajectory is simulated by an Euler--Maruyama scheme with time step $0.001$. The regulatory-frontier calculations use the grids reported below and $2,000$ simulated paths per configuration for the confidence-adjusted results. The GitHub archive contains the solver, simulation scripts, and parameter checks used to exclude infeasible configurations.

Figure \ref{fig1} shows that the incremental value generated by the issuance option is largest near $y=\underline y$ and decreases monotonically with $y$. In other words, the flexibility to recapitalize is most useful when the bank is close to the distress boundary and becomes less relevant when the bank is well capitalized. For completeness, we also plot the value function in the original $(x,\ell)$ coordinates.
\begin{figure}[H]
    \centering
    \begin{subfigure}[t]{0.32\textwidth}\centering
\includegraphics[width=\linewidth,keepaspectratio]{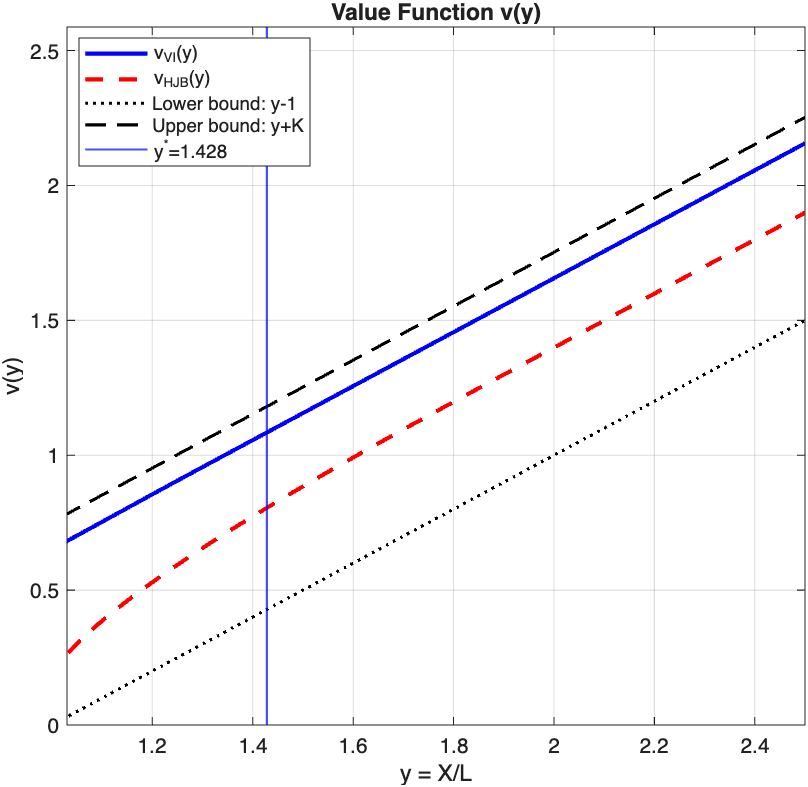}
        \caption{The value function $v(y)$}
    \end{subfigure}\hfill
    \begin{subfigure}[t]{0.32\textwidth}\centering
\includegraphics[width=\linewidth,keepaspectratio]{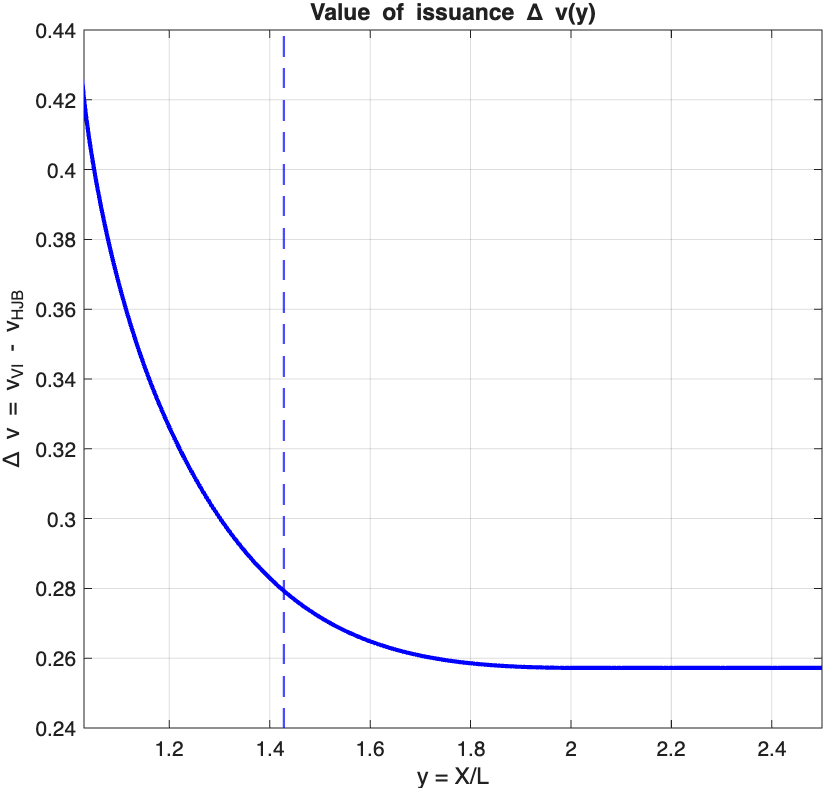}
\caption{$v(y) - v^{HJB}(y)$}
    \end{subfigure}\hfill
    \begin{subfigure}[t]{0.32\textwidth}\centering
    \includegraphics[width=\linewidth,keepaspectratio]{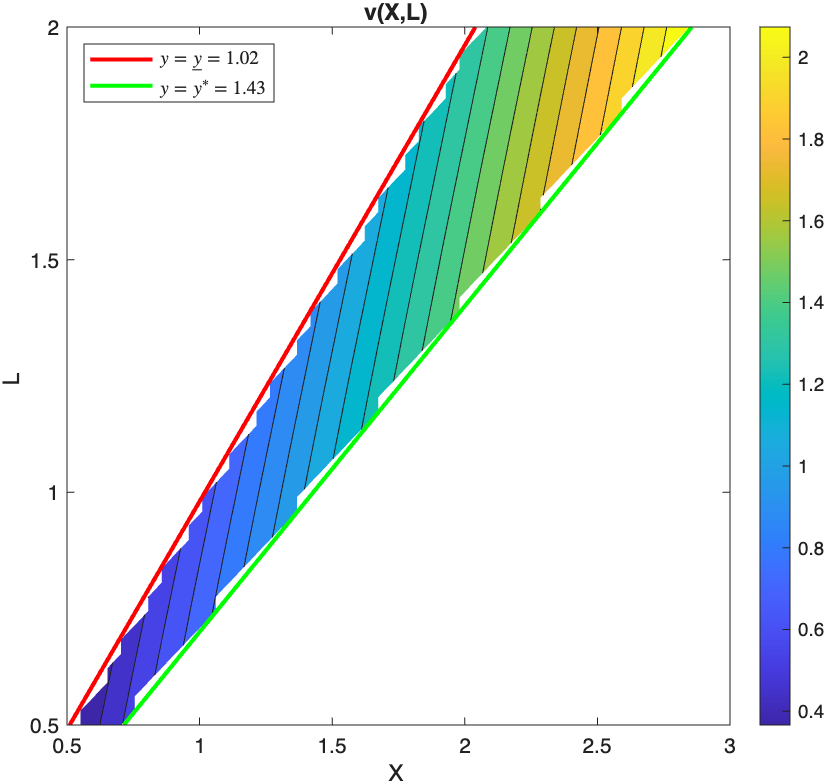}
        \caption{Contour in $(x,\ell)$}
    \end{subfigure}
    \caption{The value function in the $y$ and $(x,\ell)$ coordinates.}\label{fig1}
\end{figure}
We illustrate the strategy in Figure \ref{fig2_trajectory} by simulating a single sample path over a horizon of $50$ years, starting from $y_0 = 1.2$, with a dividend barrier $y^* = 1.4282$ and a post-issuance target $y_{post}^* = 1.1828$ (shown as $1.184$ after discretization in the figure). The path shows the expected two-sided control structure: each time the state hits the distress boundary $\underline y=1.02$, it jumps upward into the interior of the continuation region. Each time the state reaches the dividend region at $y^*$, it is reflected back into the continuation region. In this simulation, cumulative issuance is $1.509$ and cumulative dividend payments are $6.533$, as shown in Figure \ref{fig3_div}.
\begin{figure}[H]
    \centering
        \includegraphics [height=5cm,width=0.8\linewidth,keepaspectratio]{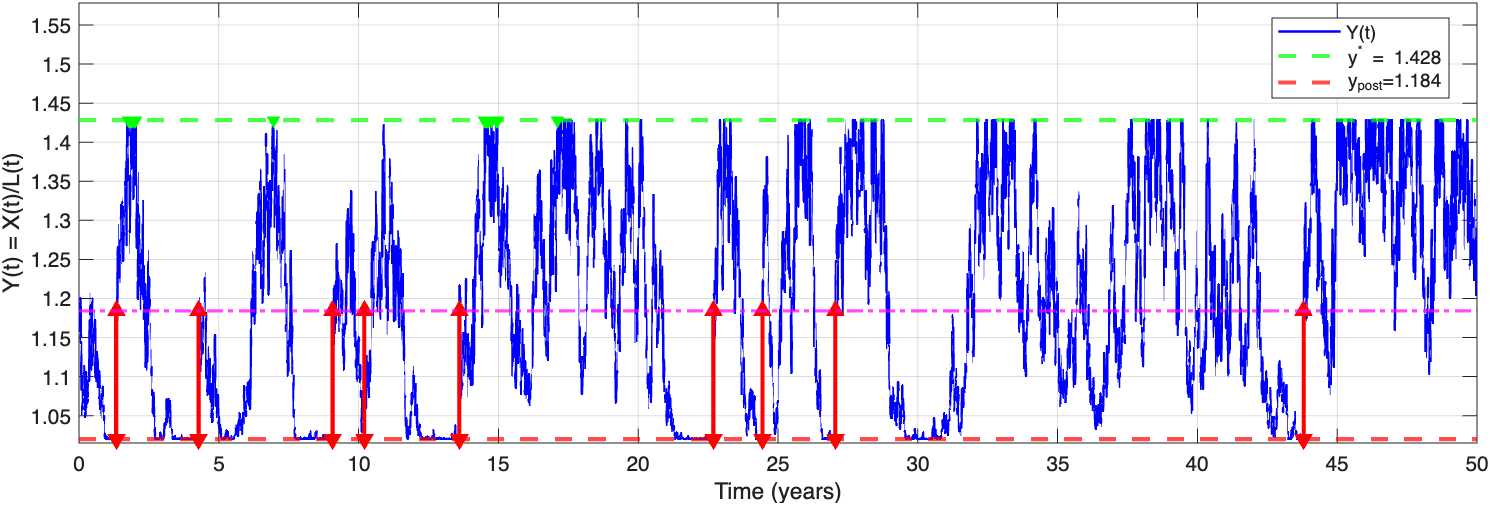}
    \caption{Trajectory from $y_0 =1.2$ over $50$ years}\label{fig2_trajectory}
\end{figure}
\begin{figure}[H]
    \centering
\includegraphics[height=5cm,width=0.7\linewidth,keepaspectratio]{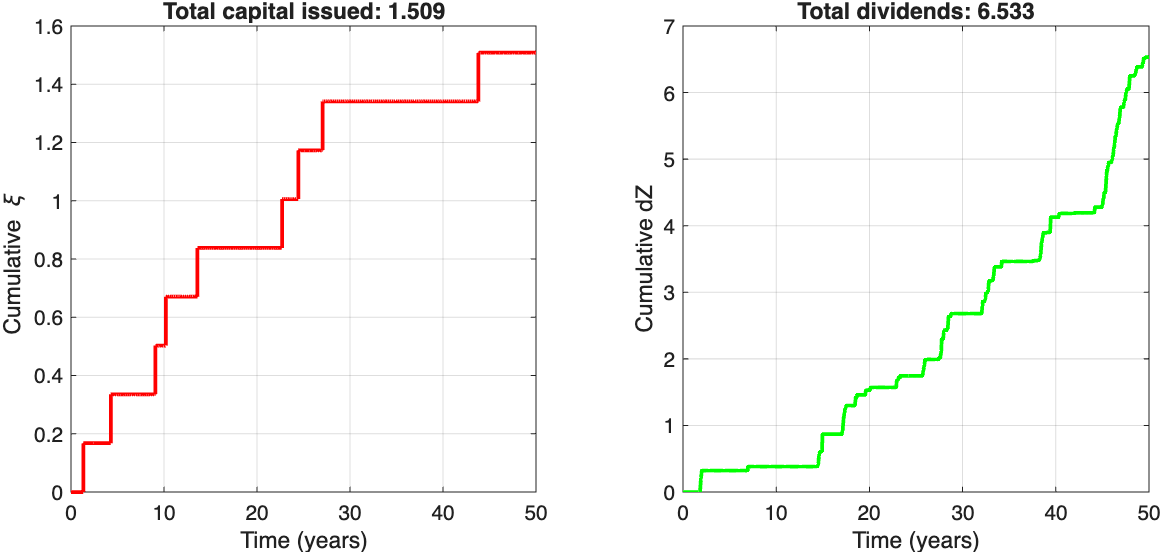}
    \caption{Cumulative capital issuance and dividends over $50$ years}\label{fig3_div}
\end{figure}
\subsection{Sensitivity to regulatory parameters}
We next study how the optimal dividend boundary $y^*$ and the relative value difference $G$ respond to changes in the regulatory parameters $a_1$, $a_2$, and $a_3$. In the model, $(a_1, a_2, a_3)$ enter only through the state-dependent investment cap $\overline{\pi}(y)$.

Changing these parameters affects the drift and volatility of the underlying process $Y$ by altering the admissible investment cap and, hence, the feasible values of $\pi(y)\leq \overline{\pi}(y)$. The dividend barrier $y^*$ reflects the balance between current dividend payouts and maintaining a future buffer under the constrained dynamics. The relative value of issuance increases when the constraint makes it more likely that the bank will reach the distress boundary, and it decreases when the process tends to avoid the lower boundary, making recapitalization less attractive. \par
The key structure is the minimum operator in the regulatory constraint, with a switching point $\widehat y>1$ that induces two distinct investment regimes. Under the baseline regulatory parameters
$(a_1,a_2,a_3)=(0.045,0.05,0.30)$,
the switching point is $\widehat y
=\frac{a_3-a_1a_2}{a_3-a_1}=
1.168$. For $\underline y\leq y\leq \widehat y$, the solvency constraint is binding, whereas for $y>\widehat y$, the liquidity constraint is binding. Under the baseline calibration, an issuance event moves the bank to the post-issuance state $y_{\mathrm{post}}^*=1.1828>\widehat y$. Thus, recapitalization not only raises the bank's capital ratio but also moves it from the solvency-dominated region into the liquidity-dominated region. The continuation value following issuance is therefore affected by both parts of the regulatory constraint. The numerical comparative statics are reported in Table \ref{sensitivity_a}. Keeping all other parameters fixed, we evaluate the value function at $y_0=1.2$ and vary one regulatory parameter at a time. We report: (i) the optimal dividend threshold $y^*$; (ii) the relative value of the issuance option,
\begin{equation}
    G(y_0)
    :=
    \frac{v(y_0)-v^{\mathrm{HJB}}(y_0)}
         {v^{\mathrm{HJB}}(y_0)},
\end{equation}
recall that $v^{\mathrm{HJB}}$ denotes the value when capital issuance is unavailable; and (iii) the shareholder value $v(y_0)$. The quantity $G(y_0)$ measures the relative value of
recapitalization, keeping $\underline y$ and $a$ fixed. It reflects both the likelihood of reaching the intervention boundary and the continuation value generated after the bank is recapitalized.
\begin{table}[H]
\centering
\caption{Sensitivity to $(a_1,a_2,a_3)$ at $y_0=1.2$.}
\label{sensitivity_a}
\small
\setlength{\tabcolsep}{3.5pt}

\begin{minipage}[t]{0.32\textwidth}
\centering
\textbf{Sensitivity to $a_1$}\\[0.3em]
\begin{tabular}{c|ccc}
\hline\hline
$a_1$ & $y^*$ & $G(1.2)$ & $v(1.2)$\\
\hline
0.045 & 1.4282 & 0.6166 & 0.8552\\
0.050 & 1.4495 & 0.6010 & 0.8469\\
0.060 & 1.4957 & 0.5670 & 0.8289\\
0.070 & 1.5472 & 0.5292 & 0.8090\\
0.080 & 1.6040 & 0.4900 & 0.7869\\
0.090 & 1.6665 & 0.4560 & 0.7626\\
0.100 & 1.7350 & 0.4274 & 0.7361\\
0.110 & 1.8102 & 0.4028 & 0.7069\\
0.120 & 1.8927 & 0.3801 & 0.6749\\
\hline\hline
\end{tabular}
\end{minipage}
\hfill
\begin{minipage}[t]{0.32\textwidth}
\centering
\textbf{Sensitivity to $a_2$}\\[0.3em]
\begin{tabular}{c|ccc}
\hline\hline
$a_2$ & $y^*$ & $G(1.2)$ & $v(1.2)$\\
\hline
0.050 & 1.4282 & 0.6166 & 0.8552\\
0.060 & 1.4237 & 0.6100 & 0.8495\\
0.080 & 1.4146 & 0.5967 & 0.8381\\
0.090 & 1.4101 & 0.5901 & 0.8324\\
0.100 & 1.4056 & 0.5834 & 0.8267\\
0.120 & 1.3965 & 0.5701 & 0.8154\\
0.150 & 1.3829 & 0.5502 & 0.7983\\
0.180 & 1.3693 & 0.5302 & 0.7812\\
0.200 & 1.3602 & 0.5168 & 0.7698\\
\hline\hline
\end{tabular}
\end{minipage}
\hfill
\begin{minipage}[t]{0.34\textwidth}
\centering
\textbf{Sensitivity to $a_3$}\\[0.3em]
\begin{tabular}{c|ccc}
\hline\hline
$a_3$ & $y^*$ & $G(1.2)$ & $v(1.2)$\\
\hline
0.240 & 1.7228 & 0.8014 & 1.0531\\
0.270 & 1.5310 & 0.6866 & 0.9380\\
0.300 & 1.4282 & 0.6166 & 0.8552\\
0.330 & 1.3616 & 0.5592 & 0.7921\\
0.360 & 1.3140 & 0.5118 & 0.7423\\
0.390 & 1.2779 & 0.4723 & 0.7019\\
0.420 & 1.2495 & 0.4391 & 0.6685\\
0.450 & 1.2266 & 0.4108 & 0.6404\\
0.480 & 1.2076 & 0.3866 & 0.6164\\
\hline\hline
\end{tabular}
\end{minipage}
\end{table}
\normalsize

First, increasing the solvency parameter $a_1$ raises the dividend threshold substantially, from $1.4282$ to $1.8927$. This result is consistent with the state-dependent nature of the solvency requirement. A larger $a_1$ reduces the admissible risky investment position most strongly at low capitalization levels. The bank responds to the tighter regulation by retaining earnings longer and paying dividends only after reaching a larger precautionary capital buffer. At the same time, shareholder value decreases from $0.8552$ to $0.6749$. The relative value of issuance also decreases with $a_1$, from $0.6166$ to $0.3801$. Although issuance restores the bank's capitalization, the recapitalized bank continues to face the more restrictive constraint, and the continuation value generated by recapitalization declines as $a_1$ increases.
Second, increasing the liquidity parameter $a_2$ lowers the dividend threshold from $1.4282$ to $1.3602$ and reduces shareholder value from $0.8552$ to $0.7698$. A larger $a_2$ shifts the liquidity cap downward, thereby reducing the bank's investment opportunities in the liquidity-dominated region, which contains both $y_0=1.2$ and $y_{\mathrm{post}}^*=1.1828$. The bank therefore begins distributing dividends earlier. The issuance gain $G(1.2)$ also decreases, from $0.6166$ to $0.5168$. Capital raised through issuance is deployed mainly in the liquidity-dominated region, so tightening $a_2$ reduces the continuation value created by the newly issued capital.
Finally, increasing $a_3$ produces a considerably stronger effect. As $a_3$ rises from $0.24$ to $0.48$, the dividend threshold falls from $1.7228$ to $1.2076$, shareholder value falls from $1.0531$ to $0.6164$, and the relative issuance gain falls from $0.8014$ to $0.3866$. Unlike $a_2$, which enters through the state-dependent term $a_2/y$, the parameter $a_3$ scales the entire liquidity constraint and determines its asymptotic upper bound:
$\lim_{y\to\infty}\overline\pi(y)=\frac{1}{a_3}$.
An increase in $a_3$ sharply reduces the productivity of retained and issued capital, causing the bank to pay dividends at a lower threshold and reducing both the total value function and the relative contribution of the issuance option.
The results reveal an important asymmetry between the two regulatory regimes. Tightening the solvency constraint near distress induces the bank to retain a larger precautionary buffer, whereas tightening the liquidity constraint at higher state levels induces earlier dividend payments. Hence, regulatory tightening does not have a uniform effect; its effect depends on the region of the state space in which the corresponding constraint binds.\par
To separate the effects of the two regulatory constraints, we compare four regimes: no Basel constraint, solvency only, liquidity only, and both constraints. Figure \ref{fig:constraint-decomposition} reports this numerical decomposition. The no-Basel case imposes only a common exogenous leverage ceiling, so the admissible investment sets are nested across regimes. Under a loose common leverage ceiling, the benchmark without a Basel constraint has the highest value, and the joint-constraint case has the lowest. The LCR-only curve is close to the joint-constraint curve, while the solvency-only curve is close to the unconstrained benchmark. Thus, the LCR accounts for most of the value loss in this parameter range. When $\pi_{\max}=1$, the LCR is everywhere weaker than the direct leverage ceiling and is therefore redundant: $v_{\text{none}} = v_{\text{LCR}}$ and $v_{\text{solvency}} = v_{\text{both}}$. In this case, the LCR imposes no additional restriction once risky investment is already tightly capped.
\begin{figure}[H]
\centering
\begin{subfigure}[t]{0.49\textwidth}
    \centering
    \textbf{\small $\pi<\infty$}\\[0.4em]
    \begin{minipage}[t][3.15cm][t]{\textwidth}
        \centering
        \small
        \resizebox{\textwidth}{!}{%
        \begin{tabular}{l|rrr}
        \hline\hline
        Regime
        & $v(1.2)$
        & Value loss
        & $G(1.2)$\\
        \hline
        None
        & 1.4040
        & 0
        & 1.1406\\

        Solvency only
        & 1.3872
        & 0.0168
        & 1.1124\\

        LCR only
        & 0.9220
        & 0.4820
        & 0.8100\\

        Both
        & 0.8552
        & 0.5488
        & 0.6166\\
        \hline\hline
        \end{tabular}%
        }
    \end{minipage}

    \vspace{0.4em}
    \includegraphics[
        width=\textwidth,
        height=6cm,
        keepaspectratio
    ]{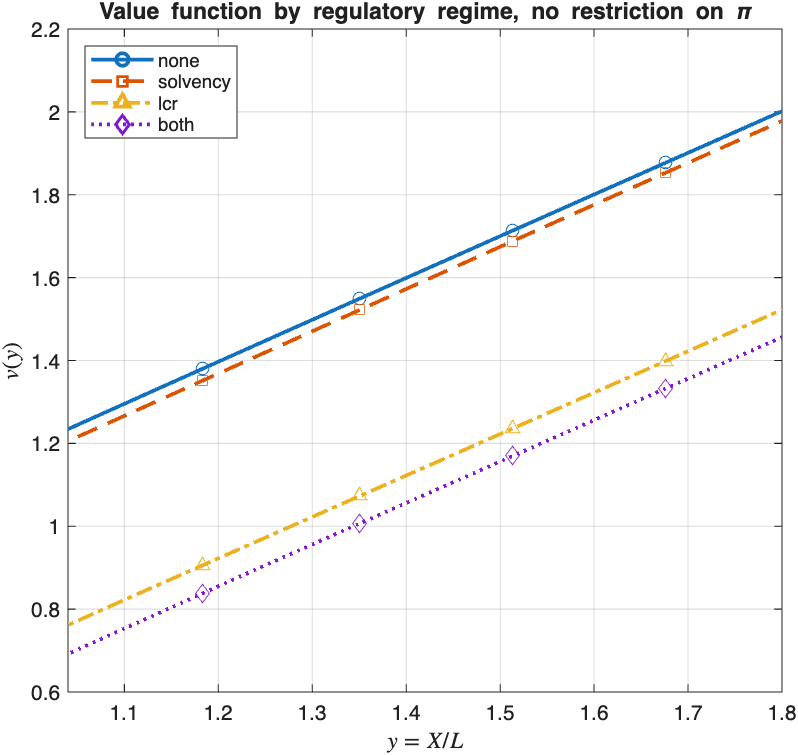}
    \caption{$\pi<\infty$.}
    \label{fig:loose-cap-regime}
\end{subfigure}
\hfill
\begin{subfigure}[t]{0.49\textwidth}
    \centering
    \textbf{\small $\pi\leq1$}\\[0.4em]
    \begin{minipage}[t][3.15cm][t]{\textwidth}
        \centering
        \small
        \resizebox{\textwidth}{!}{%
        \begin{tabular}{l|rrr}
        \hline\hline
        Effective regime
        & $v(1.2)$
        & Value loss
        & $G(1.2)$\\
        \hline
        None
        & 0.4726
        & 0
        & 0.2779\\

        Solvency only
        & 0.4527
        & 0.0199
        & 0.2295\\

        LCR only
        & 0.4726
        & 0
        & 0.2779\\

        Both
        & 0.4527
        & 0.0199
        & 0.2295\\
        \hline\hline
        \end{tabular}%
        }
    \end{minipage}
    \vspace{0.2em}
    \includegraphics[
        width=\textwidth,
        height=6cm,
        keepaspectratio
    ]{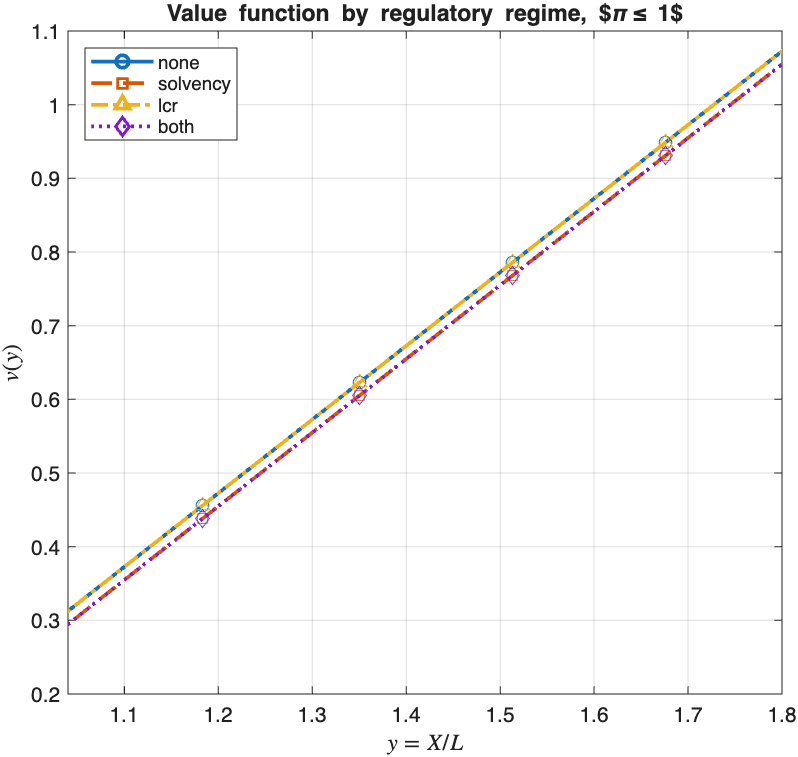}
    \caption{$\pi\leq1$.}
    \label{fig:tight-cap-regime}
\end{subfigure}
\caption{Constraint decomposition across regulatory regimes at
$y_0=1.2$.}
\label{fig:constraint-decomposition}
\end{figure}

\normalsize
\subsection{Regulatory parameter optimization}\label{regulator_optimization}
We next examine the regulator's choice of parameters $a=(a_1,a_2,a_3)$, taking into account the bank's endogenous investment, dividend, and recapitalization decisions. Throughout this subsection, the internal recapitalization threshold $\underline y$ is held fixed. For each triple $a$, the bank responds by solving the corresponding VI. Let $v_a(y)$ denote the bank's value under $a$. Since equity issuance is triggered
when the capital ratio reaches $\underline y$, we interpret reaching this boundary as a regulatory stress event and define $\tau_a:=\inf\{t\geq0:Y_t^a\leq\underline y\}$.
The associated $T$-year survival probability is $\mathbb P_{y_0}(\tau_a>T)$. Observe that the reported survival probabilities are evaluated under the original physical measure $\P$. Tighter regulatory requirements reduce the bank's admissible investment set, but their effect on resilience is not obvious. The bank may respond by changing its risky exposure, dividend threshold, or recapitalization policy. We therefore evaluate their joint effects on value and risk. For a prescribed safety level $\eta\in(0,1)$, we define the regulator's
constrained problem as
\begin{equation}
\label{eq:regulator-grid-problem}
\sup_{a\in\mathcal A_R}
\left\{
v_a(y_0):
\mathbb P_{y_0}(\tau_a>T)\geq\eta
\right\},
\end{equation}
where $\mathcal A_R$ denotes the set of admissible regulatory
parameters. Problem \eqref{eq:regulator-grid-problem} is not intended to be a complete social-welfare problem. It is a reduced-form constrained-efficiency problem: the regulator preserves bank profitability while imposing an explicit resilience requirement. A richer objective may also penalize intervention frequency, recapitalization costs, deposit-insurance losses, or other externalities. We return to such an extension after studying the constrained frontier. For the baseline comparison, we set $T=5$ and $y_0=1.2$ and consider the parameter grid
\begin{align*}
a_1&\in\{0.045,0.05,0.06,0.07,0.08,0.09,0.10,0.11,0.12\},\\
a_2&\in\{0.05,0.06,0.08,0.09,0.10,0.12,0.15,0.18,0.20\},\\
a_3&\in\{0.15,0.20,0.25,0.30,0.35,0.40,0.45,0.50,0.55\}.
\end{align*}
Parameter combinations that violate the model's feasibility conditions are removed before solving the VI. For every $a$, we compute $v_a(y_0)$ and estimate $\mathbb P_{y_0}(\tau_a>T)$ by Monte Carlo simulation. Let $\widehat p_a: = \widehat{\mathbb P}_{y_0}(\tau_a>5)$
denote the point estimate of the survival probability. A policy $a$ lies on the point-estimate Pareto frontier if there is no other
admissible policy $\widetilde a$ such that $v_{\widetilde a}(y_0)\geq v_a(y_0)$ and $\widehat p_{\widetilde a}\geq\widehat p_a$.
Thus, along the frontier, an increase in safety can be achieved only by reducing $v$. We compare two investment regimes, with and without a leverage constraint ($\pi(y)\leq 1$). Because $\pi>1$ represents a leveraged position, we treat the $\pi\leq1$ case as the main banking interpretation and the unrestricted case as a mathematical benchmark.
\begin{figure}[H]
\centering
\begin{subfigure}[t]{0.48\textwidth}
    \centering
    \includegraphics[
        width=\linewidth,
        height=0.31\textheight,
        keepaspectratio
    ]{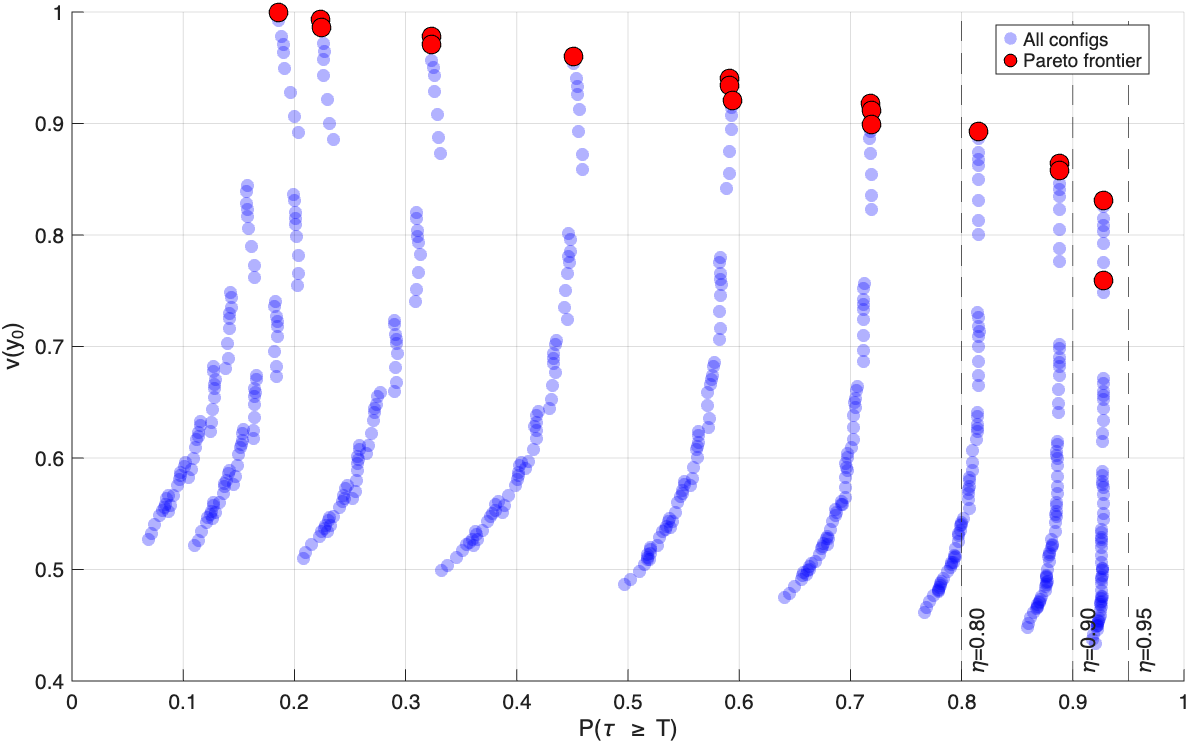}
    \caption{Pareto frontier $\pi<\infty$.}
    \label{fig:pareto-loose}
\end{subfigure}
\hfill
\begin{subfigure}[t]{0.48\textwidth}
    \centering
    \includegraphics[
        width=\linewidth,
        height=0.31\textheight,
        keepaspectratio
    ]{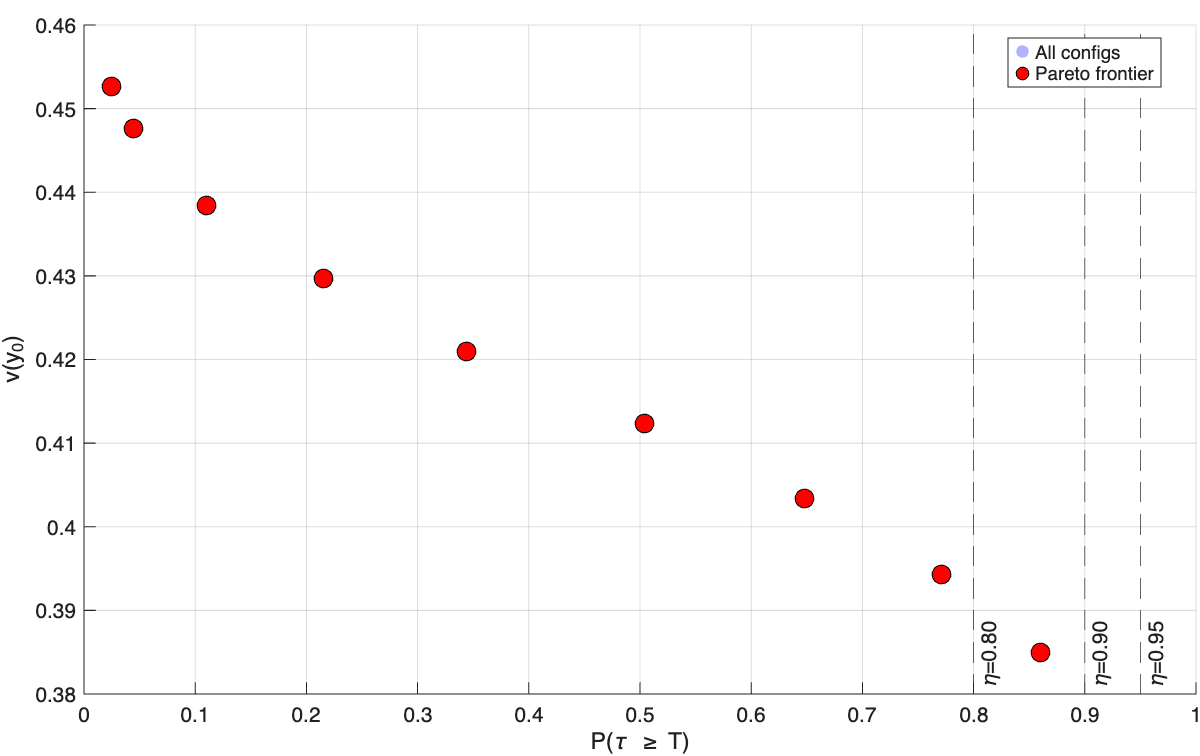}
    \caption{Pareto frontier $\pi\leq 1$.}
    \label{fig:pareto-tight}
\end{subfigure}
\caption{Frontier between shareholder value and five-year survival probability.}
\label{fig:pareto-comparison}
\end{figure}
\begin{figure}[H]
\centering
\begin{minipage}[t]{0.45\textwidth}
    \vspace{0pt}
    \centering
    \small
    \resizebox{\linewidth}{!}{%
    \begin{tabular}{ccc|ccc}
    \hline\hline
    $a_1$ & $a_2$ & $a_3$ & $y^*$ & $v(1.2)$ & $\widehat p_a$\\
    \hline
    0.045 & 0.05 & 0.25 & 1.6840 & 0.99991 & 0.1856\\
    0.050 & 0.05 & 0.25 & 1.7133 & 0.99292 & 0.2232\\
    0.060 & 0.05 & 0.25 & 1.7769 & 0.97773 & 0.3230\\
    0.070 & 0.05 & 0.25 & 1.8493 & 0.96040 & 0.4508\\
    0.080 & 0.05 & 0.25 & 1.9321 & 0.94062 & 0.5912\\
    0.090 & 0.05 & 0.25 & 2.0259 & 0.91814 & 0.7182\\
    0.100 & 0.05 & 0.25 & 2.1322 & 0.89269 & 0.8152\\
    0.110 & 0.05 & 0.25 & 2.2521 & 0.86389 & 0.8878\\
    0.120 & 0.05 & 0.25 & 2.3891 & 0.83099 & 0.9276\\
    \hline\hline
    \end{tabular}}
    \captionof{table}{Pareto frontier $\pi<\infty$.}
    \label{tab:pareto-loose-full}
\end{minipage}
\hfill
\begin{minipage}[t]{0.45\textwidth}
    \vspace{0pt}
    \centering
    \small
    \resizebox{\linewidth}{!}{%
    \begin{tabular}{ccc|ccc}
    \hline\hline
    $a_1$ & $a_2$ & $a_3$& $y^*$ & $v(1.2)$ & $\widehat p_a$\\
    \hline
    0.045 & 0.05 & 0.25& 1.0907 & 0.4527 & 0.0244\\
    0.050 & 0.05 & 0.25& 1.0964 & 0.4477 & 0.0444\\
    0.060 & 0.05 & 0.25& 1.1068 & 0.4384 & 0.1100\\
    0.070 & 0.05 & 0.25& 1.1166 & 0.4297 & 0.2152\\
    0.080 & 0.05 & 0.25& 1.1264 & 0.4210 & 0.3440\\
    0.090 & 0.05 & 0.25& 1.1361 & 0.4124 & 0.5042\\
    0.100 & 0.05 & 0.25& 1.1462 & 0.4034 & 0.6482\\
    0.110 & 0.05 & 0.25& 1.1564 & 0.3943 & 0.7708\\
    0.120 & 0.05 & 0.25& 1.1670 & 0.3849 & 0.8615\\
    \hline\hline
    \end{tabular}}
    \captionof{table}{Pareto frontier $\pi\leq 1$.}
    \label{tab:pareto-pi-one-full}
    \vspace{0.5em}
\end{minipage}
\end{figure}

Tables \ref{tab:pareto-loose-full} and \ref{tab:pareto-pi-one-full} report selected frontier points underlying Figure \ref{fig:pareto-comparison}. Without a leverage restriction, the feasible configurations generate a genuinely multidimensional set of outcomes. The scatter plot forms a sequence of approximately vertical clusters. Movement from one cluster to another is driven primarily by changes in $a_1$, whereas changes in $a_2$ and $a_3$ tend to affect $v$ more strongly than $\widehat p$. The solvency requirement $a_1$ is therefore the principal instrument
for moving the bank toward higher resilience in the present
experiment. The liquidity parameters still affect the shape
of the frontier, but additional
tightening of $a_2$ and $a_3$ often produces a visible loss in $v$ with only a small increase in the survival rate.

This pattern is even sharper under $\pi\leq1$. For our choices of $a_2\leq 0.20$, $a_3\leq 0.55$, and $y\geq\underline y=1.02$, $\frac{1}{a_3}\left(1-\frac{a_2}{y}\right)\geq\frac{1}{0.55}\left(1-\frac{0.20}{1.02}\right)\approx1.46>1$. Hence, throughout the parameter range considered, the direct ceiling $\pi\leq 1$ is tighter than the LCR constraint, so $\overline\pi_a(y)$ is independent of $a_2$ and $a_3$. Consequently, all triples sharing the same $a_1$ generate the same strategies and values. After outcome-equivalent configurations are collapsed, the frontier contains only one point for each value of $a_1$. The no-borrowing case is therefore effectively one-dimensional: only the solvency requirement remains relevant. When leveraged investment is permitted, the LCR parameters can affect the bank's optimal response, although $a_1$ remains dominant. When borrowing is prohibited, however, the LCR constraint is nonbinding within the parameter range considered.

The preceding Pareto frontier is constructed using the Monte
Carlo point estimates $\widehat p_a$ and a refined search with $20$ equally distributed values of $a_1$ on $[0.045,0.12]$, together with $a_2\in \{0.05,0.1\}$ and $a_3\in \{0.15,0.25, 0.50\}$. Let $[p_l,p_h]$ denote a $95\%$ confidence interval for $p_a$. To account for simulation error, we impose the conservative feasibility requirement $p_{l}\geq\eta$ when selecting a policy triple $a$. We also rerun the experiments on a refined grid for $a_1$ to extend the frontier. The experiment uses $2,000$ trajectories per configuration to select a candidate policy for each target. To separate policy selection from statistical validation, each
selected policy is then evaluated using an independent sample of
$5{,}000$ trajectories. The random-number generator is reset to a common seed for each candidate configuration, while the subsequent validation samples use independent seeds. Table ~\ref{tab_confidence} reports the policies selected under this confidence-adjusted criterion.
\begin{table}[H]
\centering
\small
\resizebox{\textwidth}{!}{%
\begin{tabular}{cc|ccc|ccc}
\hline\hline
Leverage regime
& $\eta$
& $a_1^*$
& $a_2^*$
& $a_3^*$
& $y^*$
& $v(1.2)$
& $\widehat p_a\;[\widehat p_l,\widehat p_h]$\\
\hline
$\pi<\infty$
& 0.80
& 0.10421
& 0.05
& 0.25
& 2.1809
& 0.88100
& $0.8434\;[0.8331,\,0.8532]$\\

$\pi<\infty$
& 0.90
& 0.12
& 0.05
& 0.25
& 2.3891
& 0.83099
& $0.9314\;[0.9241,\,0.9381]$\\
\hline
$\pi\leq1$
& 0.80
& 0.11605
& 0.05
& 0.25
& 1.1628
& 0.38860
& $0.8274\;[0.8167,\,0.8376]$\\

$\pi\leq1$
& 0.85
& 0.12
& 0.05
& 0.25
& 1.167
& 0.38492
& $0.8622\;[0.8524,0.8715]$\\
\hline\hline
\end{tabular}}
\caption{Pareto-optimal strategies with 95\% confidence intervals.}
\label{tab_confidence}
\end{table}
Under unrestricted investment, the target $\eta=0.80$ is satisfied by the regulatory vector $(a_1,a_2,a_3)\approx(0.1,0.05,0.25)$. Its estimated survival probability is $0.8434$, and the CI lower bound is $0.8331$. Increasing the target to $\eta=0.90$ leads the regulator to raise the solvency parameter to $a_1=0.12$. The resulting estimate is $0.9314$, with a lower bound of $0.9241$, so the policy satisfies the $90\%$ requirement. Under $\pi\leq1$, the target $\eta=0.80$ requires a tighter solvency constraint, $a_1=0.11605$, while $a_1=0.12$ satisfies the target $\eta=0.85$. Therefore, the parameter $a_1$ remains the main instrument for increasing survival, while $a_2$ and $a_3$ play a smaller role without a leverage restriction and no role under $\pi\leq1$ within the parameter range considered. 
\subsection{Regulatory design with a delegated intervention trigger}
\label{sec_delegation}
The regulatory analysis in the previous sections treats the intervention threshold $\underline y$ as fixed and studies the regulator's choice of the external regulatory parameters. We now ask a complementary implementation question: what changes if the external regulatory policy is fixed by the regulator, but the bank is allowed to choose its own intervention threshold?
For a fixed regulatory vector $a=(a_1,a_2,a_3)$, define $W_a(\underline y):=v_{a,\underline y}(y_0)$ and let $ \underline y^B(a)\in \arg\max_{\underline y}W_a(\underline y)$ denote the bank's preferred intervention threshold on our numerical grid. A first question is how to compare intervention outcomes when $\underline y$ itself varies. In this case, the trigger probability $\hat p^{\underline y}=\P_{y_0}(\tau_{\underline y}>T)$ is not a suitable measure across different values of $\underline y$, because a lower intervention threshold mechanically makes the boundary less likely to be reached. Instead, when a common operational constraint is required, we study the expected number of recapitalizations before time $T$, $\E[N_T]$.
    
Before studying delegation, we determine which components of the external regulatory vector can affect the bank in the relevant leverage regimes. We examine whether all three external requirements remain active under the two leverage regimes ($\pi<\infty$ and $\pi\leq1$) and across different values of $\underline y$. For each fixed pair $(a_1,\underline y)$, we measure how changing $(a_2,a_3)$ affects shareholder value and the estimated probability of avoiding intervention, $\hat p^{\underline y}$. Under $\pi<\infty$, the liquidity parameters remain economically relevant. Changing $(a_2,a_3)$ produces a maximum shareholder-value difference of approximately $0.281$, or $33\%$, and a maximum difference of $0.138$ in the probability of avoiding intervention. Thus, solvency and liquidity requirements operate as distinct regulatory instruments in this regime. Under $\pi\leq1$, varying $(a_2,a_3)$ produces no change in either
shareholder value or the probability of avoiding intervention. In particular, the $216$ configurations correspond to only $24$ distinct $(a_1,\underline y)$ pairs. All configurations sharing the same $(a_1,\underline y)$ give numerically identical outcomes when $(a_2,a_3)$ varies. The liquidity requirement is nonbinding because of the direct leverage ceiling, leaving $a_1$ as the only effective external regulatory instrument. We therefore hold $(a_2,a_3)=(0.05,0.25)$ fixed in the following $\pi\leq 1$ experiments and isolate the separate effects of the solvency requirement $a_1$ and the intervention trigger $\underline y$. Because the liquidity parameters are inactive in this regime, fixing them does not restrict the optimization over effective regulatory instruments. All experiments here use $T=5, y_0=1.2$ and $5,000$ Monte Carlo paths. Under $\pi\leq 1$ the $\underline y$ grid ranges from $1.01$ to $1.18$ with spacing $0.005$. The full $\kappa$ and $\kappa'$ sensitivity grids are documented in the archive. \par 
\subsubsection{Coordinated versus delegated intervention threshold}
\label{subsec_a1_y}
We next ask whether the bank would voluntarily choose the intervention threshold selected by a coordinated policy. For a fixed $\bar N$, we define an optimized pair $\left(a_1^C,\underline y^C\right) =  \arg\max_{a_1,\underline y}\left\{v(a_1,\underline y):\E[N_T]\leq\bar N\right\}$, where $(a_2,a_3)=(0.05,0.25)$ is fixed and $\pi_{\max}=1$. After determining $(a_1^C,\underline y^C)$, we hold the coordinated solvency requirement $a_1^C$ fixed and allow shareholders to choose their
preferred trigger, $\underline y^B\in\arg\max_{\underline y} v(a_1^C,\underline y)$. We define the value gain $\delta v: =10^4
\left(v(a_1^C,\underline y^B)-v(a_1^C,\underline y^C)\right)/v(a_1^C,\underline y^B)$, and report the results in Table \ref{tabN}. We find a clear disagreement over the intervention threshold $\underline y$: the coordinated benchmark selects a trigger between $1.010$ and $1.015$, whereas the bank chooses a trigger between $1.045$ and $1.070$. This choice is costly in terms of intervention frequency: the bank recapitalizes more often for only a modest gain in value. For example, at $\bar N = 0.5$, the bank chooses $\underline y^B=1.070$, increasing the simulated intervention count to $7.546$. The value gain, however, is only $20.25$ bp. To understand the economic significance of this change in intervention frequency, we must also examine the size of the recapitalizations.
\begin{table}[H]
\centering
\small
\setlength{\tabcolsep}{4pt}
\begin{tabular}{c|rrr|rrr}
\hline\hline
$\bar N$
& $a_1^C$
& $\underline y^C$
& $\E[N_T]$ (95\% CI)
& $\underline y^B$
& $\E[N_T]$ (95\% CI)
& $\delta v$
\\
\hline
0.50 & 0.100 & 1.015 & 0.310 [0.295,0.324] & 1.070 &  7.560 [7.442,7.651] & 20.25 \\
1.00 & 0.080 & 1.015 & 0.885 [0.861,0.909] & 1.060 & 16.109 [15.917,16.300] & 35.05 \\
2.00 & 0.060 & 1.010 & 1.596 [1.563,1.629] & 1.045 & 18.038 [17.841,18.234] & 59.04 \\
\hline\hline
\end{tabular}
\caption{Coordinated and delegated intervention thresholds.}\label{tabN}
\end{table}
We therefore hold the parameters fixed and vary only the intervention threshold. This isolates the effect of $\underline y$ on the three margins of recapitalization. Define $E_{iss} = \E[\sum_{n}e^{-\rho_L\tau_n}\xi_{n}]$ as the total discounted issuance amount. The constraint is on $\E[N_T]$, not $E_{iss}$. The later is reported as a comparative static outcome rather than the formal constraint. In the following two experiments, table $(a)$ varies the trigger at fixed external regulation $(a_1,a_2,a_3)=(0.08,0.05,0.25)$. Table $(b)$ varies the solvency requirement and evaluates outcomes at the bank-selected trigger.

\begin{table}[H]
\centering
\small

\begin{minipage}{0.72\textwidth}
\centering
\textbf{(a) Trigger effect with fixed regulation}

\vspace{0.1in}
\setlength{\tabcolsep}{4pt}
\begin{tabular}{c|rrr}
\hline\hline
$\underline y$
& $\E[N_T]$ (95\% CI)
& $\xi_{post}$
& $E_{iss}$ (95\% CI)
\\
\hline
1.010
& $0.523[0.504,0.541]$
& 0.0664
& $0.0266[0.0257,0.0275]$
\\
1.030
& $2.251[2.209,2.293]$
& 0.0462
& $0.0819[0.0804,0.0835]$
\\
1.050
& $6.382[6.296,6.468]$
& 0.0260
& $0.1318[0.1300,0.1336]$
\\
1.060
& $16.109[15.917,16.300]$
& 0.0108
& $0.1384[0.1367,0.1400]$
\\
1.080
& $34.761[34.442,35.081]$
& 0.0059
& $0.1632[0.1617,0.1648]$
\\
1.100
& $35.197[34.871,35.524]$
& 0.0061
& $0.1710[0.1694,0.1726]$
\\
\hline\hline
\end{tabular}
\end{minipage}

\vspace{0.25in}
\begin{minipage}{0.78\textwidth}
\centering
\textbf{(b) Solvency effect with bank response}
\vspace{0.1in}
\setlength{\tabcolsep}{4pt}
\begin{tabular}{c|rrrr}
\hline\hline
$a_1$
& $\underline y^B$
& $\E[N_T]$ (95\% CI)
& $\xi_{post}$
& $E_{iss}$ (95\% CI)
\\
\hline
0.045
& 1.030
& $17.465[17.272,17.658]$
& 0.0105
& $0.1460[0.1443,0.1476]$
\\
0.060
& 1.045
& $18.038[17.841,18.234]$
& 0.0107
& $0.1530[0.1513,0.1547]$
\\
0.080
& 1.060
& $16.109[15.917,16.300]$
& 0.0108
& $0.1384[0.1367,0.1400]$
\\
0.100
& 1.070
& $7.546[7.442,7.651]$
& 0.0211
& $0.1264[0.1246,0.1282]$
\\
0.120
& 1.085
& $6.573[6.473,6.673]$
& 0.0213
& $0.1107[0.1090,0.1124]$
\\
0.150
& 1.095
& $3.274[3.205,3.343]$
& 0.0265
& $0.0681[0.0667,0.0696]$
\\
\hline\hline
\end{tabular}
\end{minipage}
\caption{Frequency, size, and aggregate amount of recapitalization.}
\label{tab_recap}
\end{table}

Panel $(a)$ identifies the pure trigger effect. Raising $\underline y$ makes recapitalizations more frequent but smaller: $\xi_{post}$ falls from $0.0664$ at $\underline y=1.01$ to $0.0059$ at $\underline y=1.08$, while $\E[N_T]$ rises sharply. Expected discounted issuance also increases, from $0.0266$ at $\underline y=1.01$ to $0.1384$ at $\underline y=1.06$. A higher trigger therefore generates frequent but small recapitalizations and a greater aggregate need for external equity. Panel $(b)$ studies how recapitalization changes when $a_1$
is tightened and the bank can optimize its trigger again. From $a_1=0.06$ to $a_1=0.15$, $\E[N_T]$ falls from $18.038$ to $3.274$, while $\xi_{post}$ rises from $0.0107$ to $0.0265$ and expected discounted issuance falls from $0.153$ to $0.0681$. \par 
This comparison is particularly informative because the endogenous trigger response works in the opposite direction. Stronger solvency regulation along the bank's best-response curve yields fewer but larger recapitalizations. Even after allowing the bank to respond optimally by raising $\underline y$, a higher $a_1$ substantially reduces both recapitalization frequency and aggregate reliance on external equity. Section \ref{regulator_optimization} showed that changing $a_1$ while holding $\underline y$ fixed can substantially affect shareholder value. When the bank is allowed to optimize $\underline y$ as $a_1$ increases, however, the change in value is minimal. Thus, the solvency requirement $a_1$ remains the primary effective regulatory instrument. Optimizing $\underline y$ can nevertheless increase both recapitalization frequency and aggregate discounted issuance while reducing the size of each capital injection. Delegation is therefore inexpensive in shareholder-value terms, but it can materially affect the bank's recapitalization behavior. Whether to delegate the choice of the intervention threshold should depend on the regulator's objectives.

\subsection{Recapitalization costs and intervention discretion}
\label{subsec_kappa}
Section \ref{subsec_a1_y} showed that the bank generally prefers a higher intervention threshold than the coordinated threshold. A natural question is whether this preference is robust to the cost of recapitalization. The two issuance-cost parameters affect recapitalization differently. Recall that for a post-issuance target $z$, the amount of new equity required to move from $y$ is $\xi(y,z,\kappa, \kappa') = \frac{z-y+\kappa(y-1)}{1-\kappa'}$. Clearly, $ \frac{\partial\xi}{\partial\kappa} = \frac{y-1}{1-\kappa'}>0$ and $\frac{\partial\xi}{\partial\kappa'}=\frac{\xi}{1-\kappa'}>0$. This mechanical effect should be distinguished from the bank's optimized intervention threshold $\underline y^B$, which is the more informative object to examine. We study how this threshold varies with $\kappa$ and $\kappa'$ under $\pi\leq1$, holding $(a_2,a_3)=(0.05,0.25)$ fixed and varying $a_1$. The results are shown in Figure \ref{fig_k}, where panel $(a)$ holds $\kappa' = 0.02$ fixed and panel $(b)$ holds $\kappa=0.01$ fixed.
\begin{figure}[H]
    \centering
    \includegraphics[width=0.8\textwidth]
    {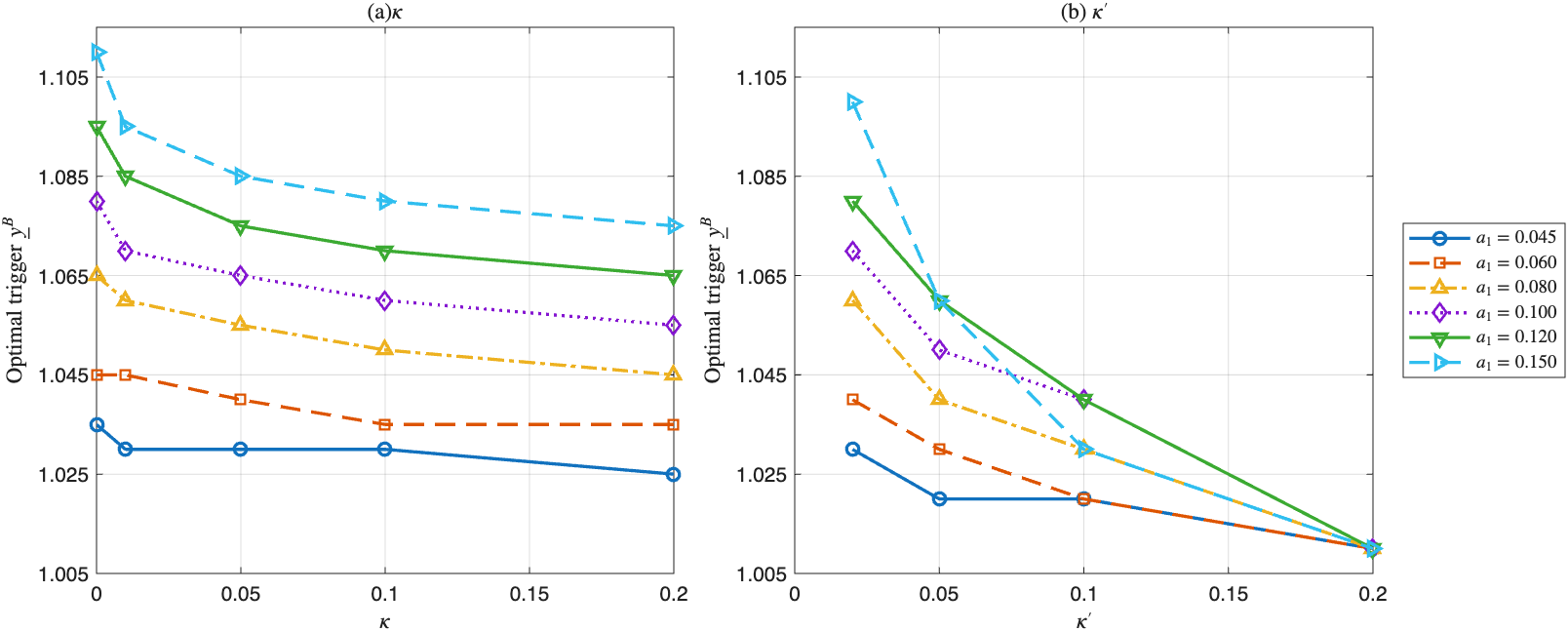}
    \caption{The preferred intervention threshold $\underline y^B$ and recapitalization costs.}
    \label{fig_k}
\end{figure}
As Figure \ref{fig_k}$(a)$ shows, the monotonicity of $\underline y^B$ in $a_1$ observed in Section \ref{subsec_a1_y} is stable across different values of $\kappa$. In addition, increasing
$\kappa$ shifts $\underline y^B$ downward for each $a_1$. Figure \ref{fig_k}$(b)$ shows a different interaction with $\kappa'$. For every fixed $a_1$, increasing $\kappa'$ lowers the bank's optimized threshold, and the response is substantially stronger than in panel $(a)$. The decline is not uniform across solvency requirements, however: the $\underline y^B$ curves cross as $\kappa'$ increases. For example, at the very high cost $\kappa'=0.20$, all six tested solvency cases select the lowest value, $\underline y^B= 1.01$, on the numerical grid. Consequently, there is no global ordering of $\underline y^B$ in $a_1$ once $\kappa'$ varies. The effect of $a_1$ on the intervention threshold is therefore conditional on relatively low recapitalization costs. For banks with higher costs, the effect of regulation is dominated by the effect of $\kappa'$. \par
Table \ref{tab_kappa} complements the figure by showing how recapitalization costs affect recapitalization frequency and aggregate external issuance. We report low, medium, and high solvency requirements and compare the baseline-cost and high-cost cases.
\begin{table}[H]
\centering
\small

\begin{minipage}{0.82\textwidth}
\centering
\textbf{$(a)$: varying $\kappa$}

\vspace{0.1in}
\setlength{\tabcolsep}{5pt}
\begin{tabular}{cc|rrr}
\hline\hline
$a_1$
& $\kappa$
& $\underline y^B$
& $\E[N_T]$ (95\% CI)
& $E_{iss}$ (95\% CI)
\\
\hline
0.045 & 0.01 & 1.030
& $17.465[17.272,17.658]$
& $0.14595[0.14433,0.14758]$
\\
      & 0.20 & 1.025
& $8.702[8.591,8.812]$
& $0.17642[0.17416,0.17868]$
\\
\hline
0.080 & 0.01 & 1.060
& $16.109[15.917,16.300]$
& $0.13837[0.13671,0.14003]$
\\
      & 0.20 & 1.045
& $4.228[4.164,4.292]$
& $0.15044[0.14815,0.15274]$
\\
\hline
0.150 & 0.01 & 1.095
& $3.274[3.205,3.343]$
& $0.06811[0.06666,0.06957]$
\\
      & 0.20 & 1.075
& $1.099[1.070,1.128]$
& $0.06558[0.06383,0.06733]$
\\
\hline\hline
\end{tabular}
\end{minipage}

\vspace{0.25in}

\begin{minipage}{0.82\textwidth}
\centering
\textbf{$(b)$: varying $\kappa'$}

\vspace{0.1in}
\setlength{\tabcolsep}{5pt}
\begin{tabular}{cc|rrr}
\hline\hline
$a_1$
& $\kappa'$
& $\underline y^B$
& $\E[N_T]$ (95\% CI)
& $E_{iss}$ (95\% CI)
\\
\hline
0.045 & 0.02 & 1.030
& $17.465[17.272,17.658]$
& $0.14595[0.14433,0.14758]$
\\
      & 0.20 & 1.010
& $5.329[5.166,5.492]$
& $0.02585[0.02505,0.02665]$
\\
\hline
0.080 & 0.02 & 1.060
& $16.109[15.917,16.300]$
& $0.13837[0.13671,0.14003]$
\\
      & 0.20 & 1.010
& $0.543[0.514,0.572]$
& $0.00509[0.00482,0.00537]$
\\
\hline
0.150 & 0.02 & 1.095
& $3.274[3.205,3.343]$
& $0.06811[0.06666,0.06957]$
\\
      & 0.20 & 1.010
& $0.0010[0,0.0020]$
& $4.37\times10^{-6}[0,8.87\times10^{-6}]$
\\
\hline\hline
\end{tabular}
\end{minipage}
\caption{Preferred intervention thresholds under different solvency requirements and recapitalization costs.}
\label{tab_kappa}
\end{table}
Table \ref{tab_kappa} shows that increasing $\kappa$ substantially reduces recapitalization frequency but does not systematically reduce total issuance. The response to $\kappa'$ is much stronger. Increasing $\kappa'$ from
$0.02$ to $0.20$ lowers the preferred intervention threshold to $1.01$ in all three solvency cases and sharply reduces both total issuance and recapitalization frequency.

The $\kappa$ and $\kappa'$ experiments clarify when the bank's choice of the intervention threshold is most important. When external recapitalization is inexpensive, the bank prefers a relatively high trigger, which leads to more frequent intervention. When recapitalization becomes more costly, the bank lowers its preferred trigger, and both recapitalization frequency and total issuance fall sharply. This response is particularly strong for the proportional cost $\kappa'$. As argued in Section \ref{subsec_a1_y}, the bank's preference for a higher intervention threshold can conflict with a regulatory objective that seeks to limit recapitalization frequency $\E[N_T]$, especially when recapitalization costs are low. This finding has a simple implication for delegation. For banks with lower recapitalization costs, regulators might retain control over the intervention threshold or restrict it when setting the regulatory requirements. For banks with sufficiently high recapitalization costs, private incentives already lead to a lower $\underline y^B$. Delegating the choice of $\underline y$ is therefore less likely to conflict
with the objective of limiting the frequency of recapitalization.
Of course, this depends on the regulator's objective. If the regulator evaluates the intervention threshold primarily through its effect on the value function, then the motivation for delegating $\underline y$ is stronger.

These results should be viewed as a first step toward a more complete Stackelberg analysis. In the present experiment, the regulator chooses the external regulatory vector, while the bank responds only through its intervention trigger on a finite grid. The next step is to formulate the leader--follower problem in continuous strategy spaces, establish the existence and regularity of the bank's response correspondence, and compare the delegated outcome with the control problem in which the regulator chooses the intervention level exogenously.
\section{Appendix}\label{appendix}
\subsection*{Proof of Proposition \ref{thm1}}
\begin{proof}
 Let $(\ell,x)\in\mathcal{S}\setminus \{0\}$ and
${\widehat{\alpha}}=((\tau_n)_{n\in\mathbb{N}^*},(\widehat{\xi}_n)_{n\in\mathbb{N}^*},\widehat{Z},\pi)\in\widehat{\mathcal{A}}.$
We define an increasing and c\`adl\`ag process by
$Z_t=\int_0^tL_u^{-1}\, d\widehat{Z}_u$ for $t\geq 0$. Define a
sequence of $\mathcal{F}_{\tau_n}$-measurable random variables by
$\xi_n=\frac{\widehat{\xi}_n}{L_{\tau_n}}$ for $n\in\mathbb{N}^*$. Using the identities
$F_{\tau_n^-} = X_{\tau_n^-} - L_{\tau_n} = (Y_{\tau_n^-}-1) L_{\tau_n}$ and $d \hat{Z}_t = L_t dZ_t$, we can write
$$
\widehat{J}^{\widehat{\alpha}}(l,x)  =
\mathbb{E}_{l,x}\Big[\int_0^{{T}^{\widehat{\alpha}}}e^{-\rho
t}L_{t}d{Z}_t-\sum_{n=1}^{\infty}e^{-\rho\tau_n}L_{\tau_n}({\xi}_n-\kappa(Y_{\tau_n^-}-1))\1_{\lbrace
\tau_n\le T^{\widehat{\alpha}}\rbrace}+e^{-\rho T^{\widehat \alpha}}L_{T^{\widehat \alpha}}(\underline y-1)\Big].$$ 
Since the process $L$ is a geometric Brownian motion, it admits the representation 
$L_t  = l\exp\left(\left(\mu_L - \frac{1}{2}\sigma_L^2\right)t + \sigma_L W_t\right)$. Define a process $M$ by
$M_t:=\frac{L_t}{l} e^{-\mu_L t} = \exp(\sigma_L W_t -\frac{1}{2}\sigma_L^2 t)$. Then $M$ is a
 positive martingale. 
For $m\geq1$, define $\theta_m:=T^{\widehat\alpha}\wedge m$ and consider the truncated payoff
\begin{align*}
\widehat J_m^{\widehat\alpha}(l,x)
:=&\E_{l,x}\Bigg[
\int_0^{\theta_m}e^{-\rho t}L_t\,dZ_t
-\sum_{n=1}^{\infty}e^{-\rho\tau_n}L_{\tau_n}
\big(\xi_n-\kappa(Y_{\tau_n^-}-1)\big)
\1_{\{\tau_n\leq\theta_m\}}\\
&+e^{-\rho T^{\widehat\alpha}}
L_{T^{\widehat\alpha}}(\underline y-1)
\1_{\{T^{\widehat\alpha}\leq m\}}
\Bigg]\\
=&l\,\E_{l,x}\Bigg[
\int_0^{\theta_m}e^{-\rho_Lt}M_t\,dZ_t
-\sum_{n=1}^{\infty}e^{-\rho_L\tau_n}M_{\tau_n}
\big(\xi_n-\kappa(Y_{\tau_n^-}-1)\big)
\1_{\{\tau_n\leq\theta_m\}}\\
&+e^{-\rho_LT^{\widehat\alpha}}
M_{T^{\widehat\alpha}}(\underline y-1)
\1_{\{T^{\widehat\alpha}\leq m\}}
\Bigg].
\end{align*}
Since $\theta_m$ is bounded, integration by parts and optional sampling give $\E\left[
\int_0^{\theta_m}e^{-\rho_Lt}M_t\,dZ_t\right]=
\E\left[M_{\theta_m}\int_0^{\theta_m}e^{-\rho_Lt}\,dZ_t\right]$, and for each $n$, \[\E\left[e^{-\rho_L\tau_n}M_{\tau_n}
\big(\xi_n-\kappa(Y_{\tau_n^-}-1)\big)\1_{\{\tau_n\leq\theta_m\}}\right]=\E\left[M_{\theta_m}e^{-\rho_L\tau_n}\big(\xi_n-\kappa(Y_{\tau_n^-}-1)\big)\1_{\{\tau_n\leq\theta_m\}}\right].\]
Moreover, on $\{T^{\widehat\alpha}\leq m\}$ we have $\theta_m=T^{\widehat\alpha}$. Hence
\begin{align*}
\widehat J_m^{\widehat\alpha}(l,x)
=l\,\E_{l,x}\Bigg[
M_{\theta_m}\Bigg(&
\int_0^{\theta_m}e^{-\rho_Lt}\,dZ_t
-\sum_{n=1}^{\infty}e^{-\rho_L\tau_n}
\big(\xi_n-\kappa(Y_{\tau_n^-}-1)\big)
\1_{\{\tau_n\leq\theta_m\}}\\
&+e^{-\rho_LT^{\widehat\alpha}}(\underline y-1)
\1_{\{T^{\widehat\alpha}\leq m\}}
\Bigg)\Bigg].
\end{align*}
We now introduce a new probability measure $\P^*$ by $\left.\frac{d\P^*}{d\P}\right|_{\mathcal F_t}=M_t$.
Since the expression in parentheses is
$\mathcal F_{\theta_m}$-measurable, optional sampling gives $\E[M_{\theta_m}H]=\E^*[H]$ for every integrable $\mathcal F_{\theta_m}$-measurable random variable $H$.
Therefore,
\begin{align*}
\widehat J_m^{\widehat\alpha}(l,x)
=l\,\E_y^*\Bigg[&
\int_0^{\theta_m}e^{-\rho_Lt}\,dZ_t
-\sum_{n=1}^{\infty}e^{-\rho_L\tau_n}
\big(\xi_n-\kappa(Y_{\tau_n^-}-1)\big)
\1_{\{\tau_n\leq\theta_m\}}\\
&+e^{-\rho_LT^{\widehat\alpha}}(\underline y-1)
\1_{\{T^{\widehat\alpha}\leq m\}}
\Bigg].
\end{align*}
Under $\mathbb{P}^*$, the processes $W^*$ and $B^*$ defined by $W_t^*:  = W_t - \sigma_L t,\quad B_t^*: = B_t - c\sigma_L t$ are standard Brownian motions with correlation coefficient $c$. We may therefore express the dynamics in terms of $(W^*, B^*)$. Let $Y^{\widehat{\alpha}}$ denote the asset-to-deposit ratio defined by
$Y^{\widehat{\alpha}}_t: =\frac{X^{\widehat{\alpha}}_t}{L_t}$.
For $\tau_n<t<\tau_{n+1}$, It\^o's formula gives
\begin{eqnarray*}
dY^{\widehat{\alpha}}_t & = &
\frac{dX^{\widehat{\alpha}}_t}{L_t}-\frac{X^{\widehat{\alpha}}_td
L_t}{(L_t)^2}+\frac{X^{\widehat{\alpha}}_td\langle
L\rangle_t}{(L_t)^3}+d\langle
X^{\widehat{\alpha}}, L^{-1}\rangle_t\\
& = & \left( Y^{\widehat{\alpha}}_t\left[(1-\pi_t)r+\pi_t\mu \right]+\gamma \right)dt + \pi_t\sigma Y^{\widehat{\alpha}}_t dB_t + \sigma_LdW_t - dZ_t\\
&&-Y^{\widehat{\alpha}}_t \mu_L dt - \sigma_L Y^{\widehat{\alpha}}_t dW_t + \sigma_L^2Y^{\widehat{\alpha}}_t dt - c\pi_t \sigma\sigma_L Y^{\widehat{\alpha}}_tdt - \sigma_L^2 dt \\
& = & \left(Y^{\widehat{\alpha}}_t\left[(1-\pi_t)r+\pi_t\mu -\mu_L + \sigma_L^2 -c\pi_t \sigma\sigma_L\right]+ \gamma - \sigma_L^2\right)dt \\
&&+ \pi_tY^{\widehat{\alpha}}_t\sigma dB_t + \sigma_L(1-Y^{\widehat{\alpha}}_t) dW_t -dZ_t.\end{eqnarray*}
By Girsanov's theorem, under the measure $\mathbb{P}^*$, the process $Y^{\widehat{\alpha}}_t$ satisfies
\begin{eqnarray*}
dY^{\widehat{\alpha}}_t & = & \left( Y^{\widehat\alpha}_t\left[(1-\pi_t)r+\pi_t\mu -\mu_L\right]+\gamma \right)dt\\
&  & +\pi_tY^{\widehat{\alpha}}_t\sigma d(B_t-c\sigma_Lt)+\sigma_L(1-Y^{\widehat{\alpha}}_t)d(W_t-\sigma_Lt)-d{Z}_t\\
& = & \left( Y^{\widehat{\alpha}}_t\left[(1-\pi_t)r+\pi_t\mu -\mu_L\right]+\gamma \right)dt\\
&  & +\pi_tY^{\widehat{\alpha}}_t\sigma dB_t^*+\sigma_L(1-Y^{\widehat{\alpha}}_t)dW_t^*-d{Z}_t
\end{eqnarray*}
for $\tau_n<t<\tau_{n+1}$.
Moreover, at issuance times $\tau_n$, we have
$X_{\tau_n}  = X_{\tau_n^-} + (1-\kappa^\prime)\widehat{\xi}_n-\kappa F_{\tau_n^-}$. Using $F_{\tau_n^-} = (Y^{\widehat{\alpha}}_{\tau_n^-} - 1)L_{\tau_n}$ and $\widehat{\xi}_n = {\xi}_nL_{\tau_n}$, this becomes
\[X_{\tau_n}  = Y^{\widehat{\alpha}}_{\tau_n^-} L_{\tau_n} + (1-\kappa^\prime){\xi}_nL_{\tau_n}-\kappa (Y^{\widehat{\alpha}}_{\tau_n^-} - 1)L_{\tau_n}. \]
Equivalently, $Y^{\widehat{\alpha}}_{\tau_n} = (1-\kappa )Y^{\widehat{\alpha}}_{\tau_n^-}+ (1-\kappa^\prime){\xi}_n+ \kappa$. Define the control set $\mathcal{A}$ by
{\small\[ {\mathcal{A}}=\{{{\alpha}} :=
((\tau_n)_{n\in\mathbb{N}^*},({\xi}_n)_{n\in\mathbb{N}^*},{Z},
\pi):\ \forall t\geq 0:\ 0\leq \pi_t\leq\overline{\pi}(Y_t^\alpha),\
\forall n\geq 1:\
{\xi}_n>\frac{\kappa}{1-\kappa^\prime}(Y^\alpha_{\tau_n^-}-1)\}. \]}
For each $\widehat{\alpha} \in \widehat{{\mathcal{A}}}$, there exists an $\alpha \in {\mathcal{A}}$ such that $Y^{\widehat{\alpha}} = Y^{\alpha}$. Finally, we observe that the admissibility condition on $\widehat{\xi}_n$ is equivalent to the corresponding condition on ${\xi}_n$: $$\{\widehat{\xi}_n>\frac{\kappa}{1-\kappa^\prime}F_{\tau_n^-}\} = \{{\xi}_n>\frac{\kappa}{1-\kappa^\prime}(Y^{\widehat{\alpha}}_{\tau_n^-}-1)\}.$$ Moreover, the intervention times coincide:
$T^\alpha=T^{\widehat{\alpha}}$. Finally, letting $m\to\infty$ and using the admissibility and integrability conditions yields
\[\widehat{J}^{\widehat{\alpha}}(\ell,x)  = \ell\mathbb{E}^*_{y}\Big[\int_0^{{T}^{\alpha}}e^{-\rho_L
t}d{Z}_t-\sum_{n=1}^{\infty}e^{-\rho_L\tau_n}({\xi}_n-\kappa(Y^{\alpha}_{\tau_n^-}-1))\1_{\lbrace
\tau_n\le T^{\alpha}\rbrace}+ e^{-\rho_L T^{\alpha}}(\underline y-1)\Big].\]
Taking the supremum over the admissible controls yields
\[\widehat v(\ell,x)=\sup_{\widehat\alpha\in\widehat{\mathcal A}}\widehat{J}^{\widehat{\alpha}}(\ell,x)=:\ell v(y)=\ell v\left(\frac{x}{\ell}\right).\]
Since $(B^*, W^*)$ under $\mathbb{P}^*$ has the same joint law as $(B, W)$ under $\mathbb{P}$, passing to the canonical space does not change the distribution of the relevant objects. To simplify notation, we therefore write $(B,W)$ instead of $(B^*,W^*)$ and understand all subsequent probabilities as taken under $\mathbb{P}^*$.
\end{proof}

\subsection*{Proof of Proposition \ref{lower_upperbound}}
We first state a supersolution comparison principle. This result will be used to construct an explicit affine upper bound for the value function.
\begin{Lemma} \label{Prop_ineq}
Let $\varphi\in\mathcal{C}^2(\lbrack \underline y,+\infty))$ such that $\varphi(\underline y)\geq \max(\underline y-1,\mathcal H \varphi(\underline y)) $ and
\beq \label{ineq_comparison} \min\left[\ \rho_L \varphi(y) - \Sup_{\pi\in\lbrack0,\overline{\pi}(y)\rbrack}\Lc^\pi \varphi(y);   \varphi^\prime(y)- 1; \varphi(y)-\mathcal{H}\varphi (y)
 \ \right] \geq 0, \quad\textrm{for every }y>\underline y.\enq Then $v \leq
\varphi$ on $\lbrack \underline y,+\infty).$
\end{Lemma}
\begin{proof}[\textbf{Proof of Lemma \ref{Prop_ineq}}]
Fix $y>\underline y$ and an arbitrary admissible control
$\alpha\in\Ac$. Let $T^{\alpha}$ denote the corresponding intervention time. Let $\tau_0=0$ and first assume $\tau_1>0$. Fix $n\in\mathbb{N}$ and work on the event $\{\tau_n< T^\alpha\}$. Note that $\{\tau_n= T^\alpha\}=\emptyset$ since $\xi_n>\frac{\kappa}{1-\kappa^\prime}(Y_{\tau_n^-}-1).$ For each integer 
$m\geq 2$, let $\theta_{m,n}$ be the first exit time from $(\underline y+1/m,m)$ after $\tau_n$: 
$$\theta_{m,n}: =\inf\{ t \geq \tau_n: Y_t^{y,\alpha} \geq m \; \mbox{ or } \;
Y_t^{y,\alpha} \leq \underline y + 1/m \}.$$
On the event $\{\tau_n< T^\alpha\}\in\mathcal{F}_{\tau_n}$, we may choose $m$ large enough so that $\tau_n<\theta_{m,n}$. Moreover, $\theta_{m,n}\nearrow T^\alpha$ almost surely as $m\to\infty$. Define $\eta := \theta_{m,n}\wedge\tau_{n+1}$. We now apply It\^o's formula to the process $\left(e^{-\rho_L
t}\varphi(Y_t^{y,\alpha})\right)$ on the time interval $[\tau_n,\eta]$:
\beqs e^{-\rho_L \eta}
\varphi(Y_{(\eta)^-}^{y,\alpha}) &=&
e^{-\rho_L \tau_{n}}\varphi(Y^{y,\alpha}_{\tau_n})
 +\int_{\tau_n}^{\eta} e^{-\rho_L t}(\Lc^\pi \varphi-\rho_L
\varphi)(Y_t^{y,\alpha})dt
\nonumber  \\
& & -  \int_ {\tau_n}^{\eta}  e^{-\rho_L t}  \varphi^\prime(Y_t^{y,\alpha}) dZ^c_t  \nonumber +  \sum_{\tau_n \leq t <  \eta} e^{-\rho_L
t}\left[ \varphi(Y_t^{y,\alpha}) - \varphi(Y_{t^-}^{y,\alpha})\right] \nonumber\\
& &   +   \int_{\tau_n}^{\eta}    e^{-\rho_L t}
\sigma \pi_t Y_t^{y,\alpha}\varphi^\prime(Y_t^{y,\alpha}) dB_t \nonumber+ \int_{\tau_n}^{\eta}    e^{-\rho_L t}
\sigma_L(1- Y_t^{y,\alpha})\varphi^\prime(Y_t^{y,\alpha}) dW_t,  \nonumber
\enqs
where $Z^c$ is the continuous part of the dividend process
$Z$. At any jump time $t\in [\tau_n,\eta)$, we have
$Y_t^{y,\alpha} - Y_{t^-}^{y,\alpha}  = -(Z_t-Z_{t^-})$. Since $\varphi'\geq1$, the mean
value theorem implies that $$\varphi(Y_t^{y,\alpha}) - \varphi(Y_{t^-}^{y,\alpha}) \leq Y_t^{y,\alpha} - Y_{t^-}^{y,\alpha}= -(Z_t-Z_{t^-})$$
for
$\tau_n$ $\leq$ $t$ $<$ $ \eta $. Since
$\varphi$ is a supersolution, we also have 
$\Lc^\pi \varphi(y) - \rho_L \varphi(y)\leq 0$
for every admissible $\pi$. Taking conditional expectations with respect to $\mathcal{F}_{\tau_n}$ and using the fact that the integrands in the stochastic integral terms are bounded by a constant depending on $m$, we obtain, on $\{\tau_n\leq T^\alpha\}$,
\beqs \E\left[
e^{-\rho_L \eta }
\varphi(Y_{( \eta)^-}^{y,\alpha})\vert \mathcal{F}_{\tau_n}\right] & \leq & e^{-\rho_L \tau_{n}}\varphi(Y^{y,\alpha}_{\tau_n}) - \E\left[
\int_{\tau_n}^{ \eta}
e^{-\rho_L t}  dZ_t^c \vert \mathcal{F}_{\tau_n}\right] \\
& & \; - \; \E\left[ \sum_{\tau_n \leq t < \eta  }
e^{-\rho_L t}(Z_t-Z_{t^-})\vert \mathcal{F}_{\tau_n} \right]. \enqs Consequently, $e^{-\rho_L \tau_{n}}\varphi(Y^{y,\alpha}_{\tau_n})  \geq  \E\left[
\int_{\tau_n}^{\eta} e^{-\rho_L t}  dZ_t +
e^{-\rho_L \eta}
\varphi(Y_{(\eta)^-}^{y,\alpha})\vert \mathcal{F}_{\tau_n}  \right]$. Since $\varphi(y)\geq 0$, letting
$m$ tend to infinity and applying Fatou's lemma yields
$$e^{-\rho_L \tau_{n}}\varphi(Y^{y,\alpha}_{\tau_n})\geq \E\left[ \int_{\tau_n}^{\tau_{n+1}\wedge T^\alpha}
e^{-\rho_L t}  dZ_t + e^{-\rho_L \tau_{n+1}\wedge T^\alpha}\varphi(Y_{(\tau_{n+1}\wedge T^\alpha)^-}^{y,\alpha})\vert \mathcal{F}_{\tau_n}\right].$$
On the event $\{\tau_{n+1}\leq T^\alpha\}$, we have
\beqs
\varphi(Y_{\tau_{n+1}^-}^{y,\alpha})  \geq \mathcal{H}\varphi(Y_{\tau_{n+1}^-}^{y,\alpha})
& \geq & \varphi((1-\kappa)Y_{\tau_{n+1}^-}^{y,\alpha}+(1-\kappa^\prime)\xi_{n+1}+\kappa)-\xi_{n+1}+\kappa(Y_{\tau_{n+1}^-}^{y,\alpha}-1)\\
& = & \varphi(Y_{\tau_{n+1}}^{y,\alpha})-\xi_{n+1}+\kappa(Y_{\tau_{n+1}^-}^{y,\alpha}-1).
\enqs
On the complementary event $\{T^\alpha<\tau_{n+1}\}$, we have
\[
Y_{(T^\alpha)^-}^{y,\alpha}\geq \underline y.
\]
Therefore, $\varphi(Y_{(\tau_{n+1}\wedge T^\alpha)^-}^{y,\alpha})\geq \underline y-1$. We then obtain
\beqs
e^{-\rho_L \tau_{n}}\varphi(Y^{y,\alpha}_{\tau_n}) & \geq & \E\left[ \int_{\tau_n}^{\tau_{n+1}\wedge T^\alpha}
e^{-\rho_L t}  dZ_t\vert \mathcal{F}_{\tau_n}\right]+
\E\left[
e^{-\rho_L T^\alpha}(\underline y-1)
\mathbf 1_{\{T^\alpha<\tau_{n+1}\}}
\,\middle|\,\mathcal F_{\tau_n}
\right]\\
& &  + \E\left[e^{-\rho_L \tau_{n+1}}\left(\varphi(Y_{\tau_{n+1}}^{y,\alpha})-\xi_{n+1}+\kappa(Y_{\tau_{n+1}^-}^{y,\alpha}-1)\right)\1_{\{\tau_{n+1}\leq T^\alpha\}}\vert \mathcal{F}_{\tau_n}\right].
\enqs
Now fix $N\geq 0$. Summing the one-step inequalities yields
\begin{align*}
  &  \varphi(Y^{y,\alpha}_0) - \E\left[ e^{-\rho_L \tau_{N+1}}\varphi(Y^{y,\alpha}_{\tau_{N+1}})\1_{\{\tau_{N+1}\leq T^\alpha\}}\right] \\
  \geq    & \sum_{n = 0}^N \left(\E\left[ \int_{\tau_n}^{{\tau_{n+1}}\wedge T^{\alpha}} e^{-\rho_L t}dZ_t - e^{-\rho_L \tau_{n+1}}\left( \xi_{n+1}-\kappa(Y_{\tau_{n+1}^-}^{y,\alpha}-1)\right)\1_{\{\tau_{n+1}\leq T^\alpha\}} \right]\right.\\
  &\qquad+\left.
\E\left[
e^{-\rho_L T^\alpha}(\underline y-1)
\sum_{n=0}^{N}
\mathbf 1_{\{\tau_n\le T^\alpha<\tau_{n+1}\}}
\right] \right).
\end{align*}
 We now distinguish two cases. Assume first that $\tau_1>0$. As $\lim_{n\to\infty}\tau_n=\infty$ and $\rho_L>0$,
$$\E\left[e^{-\rho_L T^\alpha}\sum_{n=0}^{\infty}\1_{\{\tau_n< T^{\alpha}<\tau_{n+1}\}}\right]=\E\left[e^{-\rho_L T^\alpha}\1_{\{ T^{\alpha}<\infty\}}\right]=\E\left[e^{-\rho_L T^\alpha}\right].$$
Therefore, the term $\E\left[
e^{-\rho_L\tau_{N+1}}
\varphi(Y^{y,\alpha}_{\tau_{N+1}})
\mathbf 1_{\{\tau_{N+1}\le T^\alpha\}}
\right]\geq 0$. Letting $N\to \infty$, and using the tower property together with the previous inequalities, we obtain
\beqs
\varphi(y) 
& \geq & \E\left[ \int_{0}^{T^\alpha}
e^{-\rho_L t}  dZ_t-\sum_{n=1}^{\infty}e^{-\rho_L \tau_{n}}\left(\xi_n-\kappa(Y_{\tau_{n}^-}^{y,\alpha}-1)\right)\1_{\{\tau_n\leq T^{\alpha}\}}+ e^{-\rho_L T^{\alpha}}(\underline y-1)
\right]\ .
\enqs
Now assume that $\tau_1=0$. The same arguments as before yield
\beqs
\varphi(y)
& \geq & \E\left[ \int_{0}^{T^\alpha}
e^{-\rho_L t}  dZ_t-\sum_{n=2}^{\infty}e^{-\rho_L \tau_{n}}\left(\xi_n-\kappa(Y_{\tau_{n}^-}^{y,\alpha}-1)\right)\1_{\{\tau_n\leq T^{\alpha}\}}+e^{-\rho_L T^{\alpha}}(\underline y-1)\right].
\enqs
Since $\varphi(y)\geq\mathcal{H}\varphi(y)\geq \varphi(Y_{0}^{y,\alpha})-\xi_1+\kappa(y-1)$, we obtain
$\varphi(y)  \geq J^{\alpha}(y)$ for every control $\alpha$. Taking the supremum over all controls on both sides yields $\varphi(y)\geq v(y)$ for all $y>\underline y$. It follows that
$\mathcal{H}v(\underline y)\leq \mathcal{H}\varphi(\underline y)$, and thus
$$
v(\underline y)  =  \max\left(\underline y-1,\ \mathcal{H}v(\underline y)\right)
 \leq   \max\left(\underline y-1,\ \mathcal{H}\varphi(\underline y)\right) \leq  \varphi(\underline y).$$
\end{proof}
\begin{proof}[\textbf{Proof of Proposition \ref{lower_upperbound}}]
The lower bound follows by considering the strategy that pays an immediate lump-sum dividend of size $Z_0-Z_{0-} = y-\underline y$. For the upper bound, we introduce the function $\varphi(y)=y+K$. By Lemma \ref{Prop_ineq}, it suffices to prove that $\varphi$ is a supersolution of \eqref{HJB}. Assumption \ref{assump1} implies that $\mu^*(z)\leq r+\frac{(\mu-r)^+}{\bar a}<\rho$ for all $z\geq \underline y$, so $M<\infty$ and $K$ is well defined. Since $\varphi$ is affine, we have
$$
\sup_{0\leq \pi\leq \overline{\pi}(y)}\mathcal{L}^\pi \varphi(y) =
\sup_{0\leq \pi\leq
\overline{\pi}(y)}y\left[\mu(\pi)-\mu_L\right]+\gamma  =
y\left[\mu^*(y)-\mu_L\right]+\gamma.$$ 
It follows that $\rho_L\varphi(y)-\sup_{0\leq \pi\leq \overline{\pi}(y)}\mathcal{L}^\pi \varphi(y) = y(\rho - \mu^*(y)) + \rho_L K-\gamma \geq \rho_L K-\gamma -M\geq 0$. Moreover, we have 
$\mathcal{H}\varphi(y)=\sup_{\xi>\frac{\kappa}{1-\kappa^\prime}(y-1)}\left(y+K-\kappa' \xi\right)<\varphi(y)$. Finally, since $K\geq -1$, we have $
\varphi(\underline y)=\underline y+K\ge \underline y-1$, and therefore $\varphi$ satisfies the assumptions of Proposition \ref{Prop_ineq}, which implies $v\leq \varphi$.
\end{proof}

\subsection*{Proof of Proposition \ref{viscosity}}
To prove Proposition \ref{viscosity}, we first state and prove the following lemma.
\begin{Lemma}\label{continuous}
The value function $v$ is nondecreasing and continuous on $\lbrack \underline y,+\infty)$.
\end{Lemma}
\begin{proof}[\textbf{Proof of Lemma \ref{continuous}}] We first show that $v$ is continuous on $(\underline y,\infty)$. Fix $y>\underline y$. For $\varepsilon> 0$, define the hitting times $T_{\varepsilon} \;=\; \inf \{ t\geq 0; \: Y_t^{y} = y+\varepsilon \}$ and $T_{\underline y} \;=\; \inf \{ t\geq 0; \: Y_t^{y} = \underline y. \}$. Consider a control $\alpha=((\tau_n)_{n\in\mathbb{N}^*},({\xi}_n)_{n\in\mathbb{N}^*},{Z},\pi^*)\in\Ac$ such that $Z\equiv 0$ and $\tau_1>T_{\varepsilon}$. In other words, no dividends are paid and no capital is issued before $T_{\varepsilon}$. Then $Y_t = Y^{*}_t$ for all $t < \tau_1$. On the event $\{T_{\varepsilon}<T_{\underline y}\}$, we have $Y_{T_{\varepsilon}\wedge T_{\underline y}} = y+\varepsilon$; on $\{T_{\varepsilon}\geq T_{\underline y}\}$, we have $Y_{T_{\varepsilon}\wedge T_{\underline y}} = \underline y$. Applying the dynamic programming principle (DPP) up to time $T_{\underline y}\wedge T_{\varepsilon}$, we obtain
\[v(y) \geq \E \left[ e^{-\rho_L T_{\varepsilon}\wedge T_{{\underline y}}} v(Y_{T_{\varepsilon}\wedge T_{{\underline y}}})\right] = \E \left[ e^{-\rho_L T_{\varepsilon}} v(y+\varepsilon)\1_{T_{\varepsilon}<T_{{\underline y}}}\right] + \E \left[ e^{-\rho_L T_{\underline y}} v({\underline y})\1_{T_{\varepsilon}\geq T_{\underline y}}\right].\]
Since the value function $v$ is nondecreasing, we have \beqs
 0\leq v(y+\varepsilon) - v(y) &\leq& \E \left[ (1- e^{-\rho_L
T_\varepsilon})v(y+\varepsilon) \1_{T_\varepsilon < T_{{\underline y}}} +  (v(y+\varepsilon)-v({\underline y})) \1_{T_{\varepsilon} \geq T_{{\underline y}}} \right]  \\
&\leq& v(y+\varepsilon) \left(1- \E [ e^{-\rho_L T_{\varepsilon}}]\right) +
(v(y+\varepsilon)-v({\underline y}))  \P(T_{\varepsilon} \geq T_{{\underline y}}). \enqs
Now fix $\varepsilon_0>0$, and restrict to $\varepsilon<\varepsilon_0$. Then
\begin{equation}\label{continuity_1}
\begin{aligned}
v(y+\varepsilon) - v(y) \leq{}& v(y+\varepsilon_0)
\left(1- \E [ e^{-\rho_L T_{\varepsilon}}]\right) \\
&+(v(y+\varepsilon_0) - v({\underline y}))
\P(T_{\varepsilon} \geq T_{{\underline y}}).
\end{aligned}
\end{equation}
 Now take $\varepsilon\downarrow 0$. Because $Y$ has continuous paths, $T_{\varepsilon} \longrightarrow 0$ almost surely, and by dominated convergence,
$\E [ e^{-\rho_L T_\varepsilon}] \longrightarrow  1$, as $\varepsilon \downarrow 0$. In addition, the continuity of scale functions of $Y^*$ implies that
$\P(T_\varepsilon \geq T_{{\underline y}})
\longrightarrow 0$, as $\varepsilon\downarrow 0$. 
Since $v$ is finite by Proposition \ref{lower_upperbound},
the right-hand side of \eqref{continuity_1} converges
to zero as $\varepsilon\downarrow 0 $. This proves
the right-continuity of $v$. An analogous argument, starting from $y-\epsilon$ and considering the first hitting time of $\{{\underline y},y\}$, proves the
left-continuity of the value function. We now turn to the right-continuity of $v$ at ${\underline y}$. Let $\varepsilon>0$. By immediately paying dividends down to the distress boundary, we obtain
$v({\underline y}+\varepsilon)\geq v({\underline y})+\varepsilon$. Let $\eta>0$. By \eqref{dpp}, there exists $\alpha\in\mathcal{A}$ such that
\beqs
\label{inequa_eta}
v({\underline y}+\varepsilon) & \leq & \E\Big[\int_0^{T_{\underline y}\wedge\tau_1}e^{-\rho_L t}d{Z}_t + e^{-\rho_L T_{\underline y}} v(\underline y) \1_{\{\tau_1> T_{\underline y}\}}\Big]\\
 & + & \E\Big[ e^{-\rho_L \tau_1}\Big[v\left((1-\kappa)Y_{\tau_1^-}+(1-\kappa^\prime)\xi+\kappa\right)-\xi+\kappa(Y_{\tau_1^-}-1)\Big]\1_{\{\tau_1\leq T_{\underline y}\}}\Big]+\eta.
\enqs
On the event $\{\tau_1\leq T_{\underline y}\}$, we have
$$v({\underline y})\geq \mathcal Hv({\underline y})\geq v\left(Y_{\tau_1}\right)-\tilde{\xi} + \kappa ({\underline y}-1)=v\left((1-\kappa)Y_{\tau_1^-}+(1-\kappa^\prime){\xi}+\kappa\right)-\tilde{\xi}+ \kappa ({\underline y}-1),$$
with $\tilde{\xi}=\xi+\frac{1-\kappa}{1-\kappa^\prime}(Y_{\tau_1^-}-{\underline y})$.
Since $\rho_L>0$ and $v\geq 0$, combining the previous two inequalities yields
\beqs
v({\underline y}+\varepsilon)-v({\underline y}) & \leq & \E\Big[\int_0^{(T_{\underline y}\wedge\tau_1)^-}e^{-\rho_L t}d{Z}_t\Big]\\
 & + & \E\Big[ (e^{-\rho_L \tau_1}-1)\Big[v\left((1-\kappa)Y_{\tau_1^-}+(1-\kappa^\prime)\xi+\kappa\right)\Big]\1_{\{\tau_1\leq T_{\underline y}\}}\Big]\\
 & + & \E\Big[e^{-\rho_L \tau_1}\left(\tilde{\xi}-\xi+\kappa(Y_{\tau_1^-}-1) -\kappa (\underline y -1)\right)\1_{\{\tau_1\leq T_{\underline y}\}}\Big]+\eta\\
 & \leq & \E\Big[\int_0^{(T_{\underline y}\wedge\tau_1)^-}e^{-\rho_L t}d{Z}_t\Big]+\frac{1-\kappa\kappa^\prime}{1-\kappa^\prime}\E\Big[e^{-\rho_L \tau_1}\left(Y_{\tau_1^-}-\underline y\right)\1_{\{\tau_1\leq T_{\underline y}\}}\Big]+\eta\\
 & \leq & \frac{1-\kappa\kappa^\prime}{1-\kappa^\prime}\E\Big[e^{-\rho_L T_{\underline y}\wedge\tau_1}\left(Y_{(T_{\underline y}\wedge\tau_1)^-}-\underline y\right)+\int_0^{(T_{\underline y}\wedge\tau_1)^-}e^{-\rho_L t}d{Z}_t\Big]+\eta
\enqs
 Applying It\^o's formula to the process $e^{-\rho_L t}\left(Y_{t}-\underline y\right)$ on the interval $[0,(T_{\underline y}\wedge\tau_1)^-]$ and substituting the result into the preceding inequality yields
\beqs
v({\underline y}+\varepsilon)-v({\underline y}) & \leq & \frac{1-\kappa\kappa^\prime}{1-\kappa^\prime}\E\Big[\varepsilon+\int_0^{(T_{\underline y}\wedge\tau_1)^-}e^{-\rho_L t}\left(Y_t(\mu(\pi_t)-\rho)+\gamma+\rho_L\underline y\right)\, dt\Big]+\eta.
\enqs
By the standing assumption on $\rho$, we have $\mu(\pi_t)\leq \rho$ and thus $Y_t(\mu(\pi_t) - \rho)\leq 0$. Therefore,
\[v({\underline y}+\varepsilon)-v({\underline y}) \leq  \frac{1-\kappa\kappa^\prime}{1-\kappa^\prime}\left(\varepsilon+\frac{(\gamma+\rho_L\underline y)^+}{\rho_L}\left(1-\E[e^{-\rho_L T}]\right)\right)+\eta\]
where $T=\inf\{ t\geq 0:\ \tilde{Y}_t\leq {\underline y}\}$, and $\tilde{Y}$ solves
$$\left\lbrace\begin{array}{lll}
d\tilde{Y}_t & = & \left((\mu(\pi_t)-\mu_L)\tilde{Y}_t+\gamma\right)dt+\sigma\pi_t\tilde{Y}_tdB_t+\sigma_L(1-\tilde{Y}_t)dW_t\\
\tilde{Y}_0 & = & \underline y+\varepsilon.
\end{array}
\right.$$
Let $T_\varepsilon:=\inf\{t\geq0:\widetilde Y_t\leq\underline y\}$ with $\tilde Y_0=\underline y+\varepsilon$. Choose $\delta>0$. For every $y\in[\underline y, \underline y+\delta]$ and every admissible $\pi$, we have $\sigma_Y^2(y,\pi)=\big(\pi\sigma y+c\sigma_L(1-y)\big)^2+(1-c^2)\sigma_L^2(1-y)^2\geq (1-c^2)\sigma_L^2 (\underline y-1)^2=: \underline \sigma^2>0$. Since $\overline \pi$ is bounded on $[\underline y, \underline y+\delta]$, the drift is uniformly bounded and the diffusion coefficient is uniformly non-degenerate. Standard boundary-hitting estimates then give for every $t>0$, $\sup_{\pi}\P_{\pi}^{\underline y+\epsilon}(T_{\epsilon}>t)\to 0$, and consequently $\sup_{\pi}(1-\E_{\pi}^{\underline y+\epsilon}[e^{-\rho_L T_{\epsilon}}])\to 0$. Substituting this into the preceding estimate and letting $\varepsilon,\eta\downarrow0$ proves the right-continuity of $v$ at $\underline y$.
\end{proof}

\begin{proof}[\textbf{Proof of Proposition \ref{viscosity}}]
   We first show that $v$ is a viscosity supersolution. Take $y_0\in (\underline y,\infty)$, and let $\phi \in C^2(\underline y,\infty)$ be such that $v(y_0) = \phi(y_0)$ and $v-\phi$ attains its local minimum at $y_0$. We must show that 
    $$\min \{ \rho_L \phi(y_0) - \Sup_{0\leq \pi\leq\overline{\pi}(y_0)}\Lc^\pi \phi(y_0); \phi^\prime(y_0)- 1; v(y_0)-\mathcal{H}v(y_0)\}\geq 0.$$
The inequality $v(y_0)\geq \mathcal{H}v(y_0)$ follows by considering immediate issuance. To prove $\phi'(y_0)\geq 1$, fix $\epsilon\in(0,y_0-\underline y)$. Then $v(y_0)\geq v(y_0-\epsilon)+\epsilon$. Locally, we have $\phi(y_0) - \phi(y_0-\epsilon) \geq v(y_0) - v(y_0-\epsilon) \geq \epsilon$,
which implies $\phi'(y_0)\geq 1$. It remains to show that $\rho_L \phi(y_0) -\Sup_{0\leq \pi\leq\overline{\pi}(y_0)}\Lc^\pi \phi(y_0)\geq 0$. Suppose, for contradiction, that $\rho_L \phi(y_0) -\Sup_{0\leq \pi\leq\overline{\pi}(y_0)}\Lc^\pi \phi(y_0)< 0$. Choose a strategy $\pi^*$ under which $\Lc^\pi \phi(y_0)$ attains its maximum, and consider the strategy with no dividends or issuance. Then there exists $\delta>0$ such that $\Lc^{\pi^*} \phi-\rho_L \phi>\delta$ in a neighborhood $U$ of $y_0$. For small $h>0$, define $\tau : = \inf\{t\geq 0: Y_t^{\pi^*}\notin U\}\wedge h$. By the dynamic programming principle \eqref{dpp}, $v(y_0)\geq \mathbb{E}[e^{-\rho_L \tau} v(Y^{\pi^*}_{\tau})]\geq \mathbb{E}[e^{-\rho_L \tau} \phi(Y^{\pi^*}_{\tau})]$.
Applying It\^{o}'s formula shows that the right-hand side satisfies
\[\mathbb{E}[e^{-\rho_L \tau} \phi(Y^{\pi^*}_{\tau})]  >\phi(y_0)+ \delta \mathbb{E}[\int_0^{\tau}e^{-\rho_L t}dt]>\phi(y_0) = v(y_0).
\]
This contradiction shows that $\rho_L \phi(y_0) -\Sup_{0\leq \pi\leq\overline{\pi}(y_0)}\Lc^\pi \phi(y_0)\geq0$. We next show that $v$ is a viscosity subsolution. Let $\phi\in C^2(\underline y,\infty)$ such that $v-\phi$ attains a local maximum at $y_0$ and $v(y_0) = \phi(y_0)$. We must show that 
    $$\min \{ \rho_L \phi(y_0) - \Sup_{0\leq \pi\leq\overline{\pi}(y_0)}\Lc^\pi \phi(y_0); \phi^\prime(y_0)- 1; v(y_0)-\mathcal{H}v(y_0)\}\leq 0.$$
Suppose, for contradiction, that 
$$\rho_L \phi(y_0) - \Sup_{0\leq \pi\leq\overline{\pi}(y_0)}\Lc^\pi \phi(y_0)>0,\enskip \phi^\prime(y_0)- 1>0,\enskip  v(y_0)-\mathcal{H}v(y_0)>0.$$
Then there exist $\delta>0$ and a neighborhood $U$ of $y_0$ such that for every $y\in U$, 
\[\rho_L\phi(y) -\sup_{0\leq\pi\leq\overline\pi(y)} \Lc^\pi\phi(y)\geq\delta, \quad\phi'(y)-1\geq\delta,\quad v(y)-\mathcal Hv(y)\geq\delta.\] 
Take an $\varepsilon$-optimal strategy $\alpha^{\varepsilon}$, let $Y^{\varepsilon}$ be the state process, and denote its first issuance time by $\tau_1^{\varepsilon}$. Let $\sigma:=\inf\{t\geq 0:Y_t^\varepsilon\notin U\}$ and $\theta:=\sigma\wedge\tau_1^\varepsilon\wedge h$. By DPP and It\^o's formula up to $\theta$, the drift and dividend terms
are bounded above by $-\delta\E\!\left[\int_0^\theta e^{-\rho_Lt}\,dt\right]-\delta\E\!\left[\int_0^\theta e^{-\rho_Lt}\,dZ_t\right]$.
If $\theta=\tau_1^\varepsilon$, since $Y_{\tau_1^\varepsilon-}^\varepsilon\in U$,
$v(Y_{\tau_1^\varepsilon}^\varepsilon)-\xi_1+\kappa(Y_{\tau_1^\varepsilon-}^\varepsilon-1)\leq\mathcal Hv(Y_{\tau_1^\varepsilon-}^\varepsilon)\leq v(Y_{\tau_1^\varepsilon-}^\varepsilon)-\delta $. Thus the impulse is controlled directly through the strict obstacle inequality, without evaluating $\phi$ at the post-issuance state.
Consequently, for sufficiently small $\varepsilon$, DPP yields $v(y_0)<\phi(y_0)=v(y_0)$, contradiction. Hence $v$ is a viscosity subsolution. Consequently, $v$ is a viscosity solution to the variational inequality \eqref{HJB}.
To show uniqueness, let $u$ be an upper semicontinuous viscosity subsolution and $w$ a lower semicontinuous viscosity supersolution of the variational inequality. It suffices to show that $u\leq w$. We adapt a perturbation argument similar to that in \cite{Seydel}. Let $h(x) = k_0x+C$, where $1<k_0<\frac{1}{1-\kappa'}$ and $C$ is sufficiently large. Then there exists $\tilde C>0$ such that $\min \{ \rho_L h - \Sup_{0\leq \pi\leq\overline{\pi}}\Lc^\pi h; h^\prime- 1; h-\mathcal{H}h\}>\tilde C$. Let $w^{\epsilon} = \epsilon h+(1-\epsilon )w$; then $w^{\epsilon}$ is a strict supersolution. Suppose, for contradiction, that $M^{\epsilon}:= \sup_{y\geq \underline y} (u-w^{\epsilon})(y)>0$. By the linear growth condition, $u(y)-y$ and $w(y)-y$ are bounded, then $u-w^{\epsilon}\to -\infty$ as $y\to \infty$, and the supremum is attained at some $\tilde y$. If $\tilde y=\underline y$, then $M^\epsilon=u(\underline y)-w^\epsilon(\underline y)>0$. Since $w^\epsilon(\underline y)\geq\underline y-1$, we have $u(\underline y)>\underline y-1$. The subsolution boundary condition implies $u(\underline y)\leq\mathcal Hu(\underline y)$. On the other hand, the strict supersolution property gives $w^\epsilon(\underline y)\geq\mathcal Hw^\epsilon(\underline y)+\epsilon\tilde C$. Since $u\leq w^\epsilon+M^\epsilon$, monotonicity and translation invariance of $\mathcal H$ imply $\mathcal Hu(\underline y)\leq\mathcal Hw^\epsilon(\underline y)+M^\epsilon$. Hence $M^\epsilon\leq M^\epsilon-\epsilon\tilde C$, contradiction. Thus $\tilde y>\underline y$. For $\alpha>0$, we maximize $\Phi_{\alpha}(y,z): = u(y)-w^{\epsilon}(z)-\frac{(y-z)^2}{2\alpha}$. At its maximizer $(y_{\alpha}, z_{\alpha})$, a standard argument gives $y_{\alpha}, z_{\alpha}\to \tilde y$ and $\frac{(y_{\alpha}-z_{\alpha})^2}{\alpha} \to 0$; thus, $y_{\alpha}$ and $z_{\alpha}$ are in the interior. Ishii's lemma gives $(q_{\alpha},X_{\alpha}) \in J^{2,+}u(y_{\alpha})$ and $(q_{\alpha},Y_{\alpha}) \in J^{2,-}w^{\epsilon}(z_{\alpha})$, where $q_{\alpha} = \frac{y_{\alpha}-z_{\alpha}}{\alpha}$, and
\[\begin{pmatrix}
    X_{\alpha}&0\\
    0& - Y_{\alpha}
\end{pmatrix} \leq \frac{1}{\alpha}\begin{pmatrix}
    1&-1\\
    -1& 1
\end{pmatrix} . \]
Since $w^\epsilon$ is a strict supersolution, $w^\epsilon(z_\alpha)-\mathcal Hw^\epsilon(z_\alpha)\geq\epsilon\widetilde C$ and $q_\alpha-1\geq\epsilon\widetilde C$. For the impulse branch, suppose that $u(y_\alpha)\leq\mathcal Hu(y_\alpha)$. For $\eta>0$, choose $r_\alpha>y_\alpha$ such that $u(y_\alpha)\leq
u(r_\alpha)-\frac{r_\alpha-y_\alpha}{1-\kappa'}-\frac{\kappa\kappa'}{1-\kappa'}(y_\alpha-1)+\eta$, and define $s_\alpha:=r_\alpha+z_\alpha-y_\alpha>z_\alpha$, then 
\[w^\epsilon(z_\alpha)\geq w^\epsilon(s_\alpha) -\frac{s_\alpha-z_\alpha}{1-\kappa'}-\frac{\kappa\kappa'}{1-\kappa'}(z_\alpha-1)+\epsilon\widetilde C.\]
Combining these two inequalities, we have \[
u(y_\alpha)-w^\epsilon(z_\alpha)\leq u(r_\alpha)-w^\epsilon(s_\alpha)
+\frac{\kappa\kappa'}{1-\kappa'}(z_\alpha-y_\alpha)+\eta-\epsilon\widetilde C,\]
and thus $u(r_\alpha)-w^\epsilon(s_\alpha)\leq
u(y_\alpha)-w^\epsilon(z_\alpha)$ by maximality of $\Phi_\alpha$. Thus $0\leq
\frac{\kappa\kappa'}{1-\kappa'}
|y_\alpha-z_\alpha|
+\eta-\epsilon\widetilde C$, first take $\eta<\frac{\epsilon\widetilde C}{4}$ and take $\alpha$ small enough that $\frac{\kappa\kappa'}{1-\kappa'}
|y_\alpha-z_\alpha|<\frac{\epsilon\widetilde C}{4}$, and this yields a contradiction. Hence $u(y_\alpha)-\mathcal Hu(y_\alpha)>0$. Since $q_{\alpha}-1\geq \epsilon \tilde C$, the differential operator binds for both functions. Let $\pi_\alpha$ maximize the Hamiltonian at $y_\alpha$ and set
$\widehat\pi_\alpha :=\min\{\pi_\alpha,\overline\pi(z_\alpha)\}$. Then $\widehat\pi_\alpha$ is admissible at $z_\alpha$. Since
$\overline\pi$, the drift and diffusion coefficients are locally
Lipschitz, define $b(y,\pi): = y(\mu(\pi)-\mu_L)+\gamma$, then
\[|\pi_\alpha-\widehat\pi_\alpha|+|b(y_\alpha,\pi_\alpha)-b(z_\alpha,\widehat\pi_\alpha)|+|\sigma_Y(y_\alpha,\pi_\alpha)-\sigma_Y(z_\alpha,\widehat\pi_\alpha)|\leq C|y_\alpha-z_\alpha|.
\]
Combining the subsolution and strict supersolution inequalities with Ishii's matrix inequality, we have $\rho_L(u(y_\alpha)-w^\epsilon(z_\alpha))\leq C\frac{|y_\alpha-z_\alpha|^2}{\alpha}-\epsilon\widetilde C$. Since $|y_\alpha-z_\alpha|^2/\alpha\to0$, taking $\alpha\downarrow0$ yields $\rho_LM^\epsilon\leq-\epsilon\widetilde C$, a contradiction. Hence $u\leq w^\epsilon$, and letting $\epsilon\downarrow 0 $ completes the proof.
\end{proof}

\subsection*{Proof of Proposition \ref{concavity}}
\begin{proof}
We show that $v$ is concave in three steps. First, define $v^0$ as the value function when capital issuance is not allowed; that is,
$v^0(y)=\sup_{{{\alpha}}\in\mathcal{A}}\mathbb{E}_{y}\Big[\int_0^{T^{{\alpha}}}e^{-\rho_L
t}d{Z}_t+e^{-\rho_L
T^{{\alpha}}}(\underline y-1) \Big]$. Here, the control set is
${\mathcal{A}}=\{{{\alpha}} :=
({Z},
\pi):\ \forall t\geq 0:\ 0\leq \pi_t\leq\overline{\pi}(Y_t)\}$.
Under such a control, the state process $Y$ satisfies 
$$dY^{\alpha}_t =\left( Y_t\left[\mu\pi_t +r(1-\pi_t)-\mu_L\right]+\gamma\right)dt +\sigma \pi_t Y_t dB_t +\sigma_L (1-Y_t)dW_t -d Z_t,$$
with $0\leq \pi_t\leq \overline \pi (Y_t) = \min\left(\frac{1}{a_1}(1-\frac{1}{Y_t});\ \frac{1}{a_3}(1-\frac{a_2}{Y_t})\right)$. Define the control process $U_t: = \pi_t Y_t$ for all $t\geq 0$.
Then the dynamics can be written as 
$$dY^{\alpha}_t =\left( (r-\mu_L)Y_t +(\mu-r) U_t +\gamma\right)dt +\sigma U_t dB_t +\sigma_L (1-Y_t)dW_t -d Z_t,$$
and the admissible control set becomes 
${\mathcal{A}}(y)=\{{{\alpha}} :=
({Z},U):\ \forall t\geq 0:\ 0\leq U_t\leq\overline{u}(Y_t), Y_0 = y\}$
with $\overline{u}(y) = \min\left(\frac{y-1}{a_1},\ \frac{y-a_2}{a_3}\right)$, which is concave in $y$. Therefore, the set $\{(y,U): y\geq \underline y, 0\leq U\leq \overline u(y)\}$ is convex, and the coefficients are affine in the state. Applying the DPP before absorption and restarting from the continuation value yields the concavity of $v^{0}$, as in \cite{chotakzho03}. Next, we show that the operator
$$\mathcal{H} \varphi (y)= \sup_{\xi>\frac{\kappa}{1-\kappa^\prime}(y-1)}\Big[ \varphi((1-\kappa)y+(1-\kappa^\prime)\xi+\kappa)-\xi+\kappa(y-1)\Big]$$
preserves concavity; that is, $\mathcal{H} \varphi $ is concave whenever $\varphi$ is concave. Let
$\xi_{\lambda}: = \lambda \xi_1+(1-\lambda)\xi_2$,
where $\xi_i>\frac{\kappa}{1-\kappa^\prime}(y_i-1)$, $i=1,2$. Then $\xi_{\lambda}$ is feasible for $y_{\lambda}$. Define $F(y,\xi) := (1-\kappa)y+(1-\kappa^\prime)\xi+\kappa$. Since $F$ is affine in $(y,\xi)$,
\begin{align*}
    \mathcal{H} \varphi (y_{\lambda})     &\geq \varphi(F(y_{\lambda},\xi_{\lambda}))-\xi_{\lambda}+\kappa(y_{\lambda}-1)\\
    & = \varphi(\lambda F(y_{1},\xi_{1})+(1-\lambda) F(y_{2},\xi_{2}))-\xi_{\lambda}+\kappa(y_{\lambda}-1)\\
    &\geq \lambda \varphi(F(y_{1},\xi_{1}))+(1-\lambda )\varphi(F(y_{2},\xi_{2}))-\xi_{\lambda}+\kappa(y_{\lambda}-1)\\
    & = \lambda \left(\varphi(F(y_{1},\xi_{1}))-\xi_{1}+\kappa(y_{1}-1)\right) +(1-\lambda )\left(\varphi(F(y_{2},\xi_{2}))-\xi_{2}+\kappa(y_{2}-1)\right).
\end{align*}
Taking the supremum over $\xi_1$ and $\xi_2$ on the right-hand side yields
$\mathcal{H} \varphi (y_{\lambda})   \geq \lambda\mathcal{H} \varphi (y_{1})+(1-\lambda )\mathcal{H} \varphi (y_{2})$.
In other words, $  \mathcal{H}$ preserves concavity. We now define a sequence of functions $\{v^{(n)}\}_{n\geq 0}: [\underline y,\infty) \to [0, \infty)$ by setting 
$v^{(0)} = v^0$, $ v^{(n+1)} = T (\mathcal{H}v^{(n)})$ for $ n\geq 0$,
where the continuation operator $T$ is defined by
\[\begin{aligned}
(Tf)(y):={}&\sup_{(Z,U,\tau)} \mathbb{E}_{y}\Bigg[
\int_0^{\tau\wedge T_{\underline y}} e^{-\rho_L t}dZ_t
+e^{-\rho_L \tau} f(Y_{\tau})\1_{\{\tau< T_{\underline y}\}} \\
&\qquad\qquad+e^{-\rho_L T_{\underline y}}
\max(\underline y-1, f(\underline y))
\1_{\{\tau \geq T_{\underline y}\}}\Bigg].
\end{aligned}
\]
The same argument used for $v^0$ shows that $T$ preserves concavity, and both $T$ and $\mathcal H$ are monotone. By induction, each $v^{(n)}$ is concave and $v^{(n)}\leq v^{(n+1)}\leq v$. 

Fix $\epsilon>0$ and let $\alpha$ be an $\epsilon-$optimal admissible strategy. Let $\alpha^{(n)}$ agree with $\alpha$ in the first $n$th and after that we do not issue. Since $y-1\leq v^0(y)\leq v(y)\leq y+K$, the loss from using this suboptimal strategy is bounded by $(K+1)\E[e^{-\rho_L\tau_n}1_{\tau_n<T^{\alpha}}]$. Since $\tau_n\uparrow \infty$ and $\rho_L>0$, this bound goes to $0$ by dominated convergence. Thus, $v(y)-\epsilon\leq \liminf_nv^{(n)}(y)$. Letting $\epsilon\downarrow 0$ gives $v=\sup_{n\geq 0} v^{(n)}$. Therefore, $v^{(n)}\uparrow v$, and the pointwise limit $v$ is concave. 
\end{proof}

\subsection*{Proof of Corollary \ref{regularity}}
\begin{proof}
    Since $v$ is concave on $[\underline y,\infty)$, its right and left derivatives exist for every $y>\underline y$. Define
$$D^+ v(y): =\lim_{h\downarrow 0} \frac{v(y+h)-v(y)}{h},\quad D^- v(y): =\lim_{h\uparrow 0} \frac{v(y+h)-v(y)}{h}. $$
These derivatives are finite and non-increasing in $y$. An immediate dividend payment of size $\epsilon$ yields $v(y+\epsilon)\geq v(y)+\epsilon$, and thus $D^+ v\geq 1$.
The continuation set can be written as $\Cc= \{y>\underline y:  v(y)>\mathcal{H}v(y), \enskip D^+v(y)>1\}$.
Recall that $D^+$ is right-continuous and $D^-$ is left-continuous. The linear upper bound on the value function implies that $y\mapsto \mathcal{H} v(y)$ is upper semicontinuous and that the set $\mathcal O: =\{y>\underline y: v(y)>\mathcal{H}v(y)\}$ is open. Suppose that $v$ is not differentiable at some $y_0\in \mathcal O$. Then $D^- v(y_0)>D^+ v(y_0)$. Take $k\in (D^+ v(y_0), D^- v(y_0))$ and define a test function $\tilde v (y) = v(y_0)+k(y-y_0)-N(y-y_0)^2$ that touches $v$ from above for some large $N>0$. Then $\rho_Lv - \sup_{\pi}\mathcal{L}^{\pi }\tilde v\leq 0$. However, $\rho_Lv - \sup_{\pi}\mathcal{L}^{\pi }\tilde v\to \infty$ as $N\to \infty$, a contradiction. Thus, $v\in C^1(\mathcal{O})$ and $D^+v = v'$ on $\mathcal{O}$, which implies that $\Cc$ is open. On $\Cc$, the variational inequality reduces to the HJB equation $\rho_L v = \sup_{\pi} \Lc^{\pi} v$. Since this equation is locally uniformly elliptic on $\Cc$, the value function satisfies $v\in C^2(\Cc)$, and the HJB equation holds there in the classical sense by the arguments in \cite{soner}. We next turn to the issuance set $\Kc$.
Fix $y\in\Dc$. Since $D^+v(y) = 1$, we have $v(z)\leq v(y)+(z-y)$ for all $z>y$. Let $z = (1-\kappa) y+(1-\kappa')\xi +\kappa$. Then
$v(z)-\xi+\kappa(y-1)\leq v(y) - \kappa' \xi$.
Taking the supremum over admissible $\xi$ on both sides yields
$$\mathcal{H}v(y)\leq v(y) - \kappa' \inf_{\xi>\frac{\kappa}{1-\kappa'} (y-1)}\xi = v(y) - \frac{\kappa \kappa'}{1-\kappa'}(y-1)<v(y)$$
since $\kappa, \kappa'>0$ and $y>\underline y$ on $\Dc$. This implies that $\Kc \cap \Dc = \emptyset$, and thus every interior point of $\Dc$ belongs to $\mathcal{O}$ and $v$ is affine there. Therefore, $v \in C^1(\Cc\cup {\rm int}(\Dc) )$. Since $Y$ is non-degenerate at $y^*>\underline y>1$, $v$ is $C^1$ across the dividend boundary.
\end{proof}

\subsection*{Proof of Theorem \ref{optimal_strategy}}
\begin{proof}
\textbf{(i).} By concavity, $D^+v$ is non-increasing and $D^+v\geq1$. The bound $y-1\leq v(y)\leq y+K$ implies $D^+v(y)\downarrow1$ as $y\to\infty$. If $D^+v(y)>1$ for all $y$, then for all sufficiently large $y$ we have $v(y)>\mathcal Hv(y)$ and the HJB equation holds. Using $v''\leq0$ and letting $y\to\infty$ gives $\rho\leq\sup_{y\geq\underline y}\mu^*(y)$, contradicting Assumption \ref{assump1}. Hence $\Dc\neq\emptyset$, and concavity implies $\Dc=[y^*,\infty)$ for some finite $y^*$. Moreover, smooth fit holds at $y = y^*$. Consequently, $v$ is affine with slope $1$ on the entire dividend region:
\[\rho_Lv(y)-\sup_{0\leq \pi\leq \overline{\pi}(y)}\Lc^\pi v(y)
=(\rho-\mu^*(y))y+\rho_L(v(y^*)-y^*)-\gamma\geq 0.\]
Letting $y\downarrow y^*$ gives $\rho_Lv(y^*)\geq \gamma +(\mu^*(y^*)-\mu_L)y^*$. Conversely, as shown in $(ii)$ below, issuance is not active in a neighborhood to the left of $y^*$. Taking $y_n\uparrow y^*$ with $y_n\in \Cc$, concavity gives $\rho_L v(y_n)\leq v'(y_n)(\gamma +(\mu^*(y_n)-\mu_L)y_n)$. Letting $n\to\infty$ and using smooth fit yields the other inequality, and thus $\rho_Lv(y^*)= \gamma +(\mu^*(y^*)-\mu_L)y^*$.\par 
\textbf{(ii).} 
We first show that issuance cannot occur in the interior $(\underline y, \infty)$. Fix $y\in\lbrack \underline y,y^*)$ such that $1<D^+ v(y)\leq \frac{1}{1-\kappa^\prime}$. For every $z>y$, concavity gives $v(z)\le v(y)+D^+v(y)(z-y)$. Therefore, $\mathcal Hv(y)
\leq v(y)-\frac{\kappa\kappa'}{1-\kappa'}(y-1)<v(y)$. Hence, such a point cannot belong to $\Kc$. Consequently, $\Kc\cap(\underline y,\infty)
\subset
\{y>\underline y:\ D^+v(y)>\frac{1}{1-\kappa'}\}$. 
Now suppose that $\Kc\cap(\underline y,\infty)\not=\emptyset$. Let $y_K\in(\underline y,y^*]\cap\Kc$. There exists $\xi_K\in(\frac{\kappa}{1-\kappa^\prime}(y_K-1),\infty)$ such that
$v\left(\overline{y}_K\right)-\xi_K+\kappa(y_K-1)=v(y_K)$,
where $\overline{y}_K=(1-\kappa)y_K+(1-\kappa^\prime)\xi_K+\kappa$. Fix $y\in(\underline y,y_K)$, and choose $\xi=\frac{1}{1-\kappa^\prime}\left(\overline{y}_K-(1-\kappa)y-\kappa\right).$ With this choice of $\xi$, we obtain
\beqs
\Hc v(y) & \geq & v\Big[(1-\kappa)y+(1-\kappa^\prime)\xi+\kappa\Big]-\xi+\kappa(y-1)\\
 & = & v(y_K)-\frac{1-\kappa\kappa^\prime}{1-\kappa^\prime}(y_K-y).
\enqs
On the other hand, the concavity of $v$ implies that
$$v(y)\leq v(y_K)-(y_K-y)D^- v(y_K)<v(y_K)-\frac{1-\kappa\kappa^\prime}{1-\kappa^\prime}(y_K-y),$$
since $D^- v(y_K)\geq D^+ v(y_K)>\frac{1}{1-\kappa^\prime}$.
This is a contradiction, since we have shown $v(y)\geq\Hc v(y)> v(y)$. Hence, if $\Kc\neq \emptyset$, then $\Kc=\{ \underline y\}$. Define $\psi(z) = v(z) - \frac{z-\underline y}{1-\kappa'} -\frac{\kappa\kappa'}{1-\kappa'}(\underline y-1)$. Since $v$ has linear growth, $\lim_{z\to \infty}\psi(z) = -\infty$, and $\lim_{z\downarrow \underline y}\psi(z)\leq v(\underline y)$. Therefore, $\psi$ attains its maximum at some interior point $y^*_{post} >\underline y$. The first-order condition gives $v'(y^*_{post}) = \frac{1}{1-\kappa'}$. Finally, value matching gives $v(\underline y) = \mathcal{H}v(\underline y) = v({y_{post}^*})-\frac{y_{post}^*-\underline y}{1-\kappa^\prime} - \frac{\kappa\kappa'}{1-\kappa'}(\underline y-1)$.
\end{proof}
\bibliographystyle{plain}
\bibliography{references}

@techreport{baco13,
  author      = {{Basel Committee on Banking Supervision}},
  title       = {{Basel III}: {The Liquidity Coverage Ratio and liquidity risk monitoring tools}},
  institution = {Bank for International Settlements},
  number      = {BCBS 238},
  year        = {2013},
  month       = jan,
  url         = {https://www.bis.org/publ/bcbs238.pdf}
}

@article{Alvarez2008,
  author = {Alvarez, Luis H. R.},
  title = {A Class of Singular Stochastic Control Problems with Optimal Stopping and Impulse Control},
  journal = {Stochastic Processes and their Applications},
  volume = {118},
  number = {8},
  pages = {1423--1452},
  year = {2008},
  doi = {10.1016/j.spa.2007.10.004},
}

@article{chotakzho03,
  author  = {Choulli, Tahir and Taksar, Michael and Zhou, Xun-Yu},
  title   = {A diffusion model for optimal dividend distribution for a company with constraints on risk control},
  journal = {SIAM Journal on Control and Optimization},
  volume  = {41},
  number  = {6},
  pages   = {1946--1979},
  year    = {2003},
  doi     = {10.1137/S0363012900382667}
}

@article{decvil05,
  author  = {D{\'e}camps, Jean-Paul and Villeneuve, St{\'e}phane},
  title   = {Optimal dividend policy and growth option},
  journal = {Finance and Stochastics},
  volume  = {11},
  number  = {1},
  pages   = {3--27},
  year    = {2007},
  doi     = {10.1007/s00780-006-0027-z}
}

@article{guotom08,
  author  = {Guo, Xin and Tomecek, Pascal},
  title   = {Connections between singular control and optimal switching},
  journal = {SIAM Journal on Control and Optimization},
  volume  = {47},
  number  = {1},
  pages   = {421--443},
  year    = {2008},
  doi     = {10.1137/060669024}
}

@article{jeashi95,
  author  = {Jeanblanc-Picqu{\'e}, Monique and Shiryaev, Albert N.},
  title   = {Optimization of the flow of dividends},
  journal = {Russian Mathematical Surveys},
  volume  = {50},
  number  = {2},
  pages   = {257--277},
  year    = {1995},
  doi     = {10.1070/RM1995v050n02ABEH002054}
}

@article{lyphvi07,
  author  = {Ly Vath, Vathana and Pham, Huy{\^e}n and Villeneuve, St{\'e}phane},
  title   = {A mixed singular/switching control problem for a dividend policy with reversible technology investment},
  journal = {The Annals of Applied Probability},
  volume  = {18},
  number  = {3},
  pages   = {1164--1200},
  year    = {2008},
  doi     = {10.1214/07-AAP482}
}

@article{Peukep06,
  author  = {Peura, Samu and Keppo, Jussi},
  title   = {Optimal Bank Capital with Costly Recapitalization},
  journal = {The Journal of Business},
  volume  = {79},
  number  = {4},
  pages   = {2163--2202},
  year    = {2006},
  month   = jul,
  doi     = {10.1086/503660}
}

@article{soner,
  author  = {Shreve, Steven E. and Soner, H. Mete},
  title   = {Optimal investment and consumption with transaction costs},
  journal = {The Annals of Applied Probability},
  volume  = {4},
  number  = {3},
  pages   = {609--692},
  year    = {1994},
  doi     = {10.1214/aoap/1177004966}
}

@article{Seydel,
  author  = {Seydel, Roland C.},
  title   = {Existence and uniqueness of viscosity solutions for {QVI} associated with impulse control of jump-diffusions},
  journal = {Stochastic Processes and their Applications},
  volume  = {119},
  number  = {10},
  pages   = {3719--3748},
  year    = {2009},
  doi     = {10.1016/j.spa.2009.07.004}
}

@article{soto11,
  author  = {Sotomayor, Luis R. and Cadenillas, Abel},
  title   = {Classical, singular, and impulse stochastic control for the optimal dividend policy when there is regime switching},
  journal = {Insurance: Mathematics and Economics},
  volume  = {48},
  number  = {3},
  pages   = {344--354},
  year    = {2011},
  doi     = {10.1016/j.insmatheco.2011.01.002}
}

@article{HM,
  title={Bank capital, liquid reserves, and insolvency risk},
  author={Hugonnier, Julien and Morellec, Erwan},
  journal={Journal of Financial Economics},
  volume={125},
  number={2},
  pages={266--285},
  year={2017},
  publisher={Elsevier}
}

@article{bolton,
  title={Dynamic banking and the value of deposits},
  author={Bolton, Patrick and Li, Ye and Wang, Neng and Yang, Jinqiang},
  journal={The Journal of Finance},
  volume={80},
  number={4},
  pages={2063--2105},
  year={2025},
  doi={10.1111/jofi.13454},
  publisher={Wiley Online Library}
}

@misc{boltonIA,
  title={Internet Appendix to ``Dynamic banking and the value of deposits''},
  author={Bolton, Patrick and Li, Ye and Wang, Neng and Yang, Jinqiang},
  year={2025},
  url={https://yeli-macrofinance.com/dynamic_bank.pdf},
  note={Internet Appendix}
}

@article{bcw2011,
  title={A unified theory of {Tobin's q}, corporate investment, financing, and risk management},
  author={Bolton, Patrick and Chen, Hui and Wang, Neng},
  journal={The Journal of Finance},
  volume={66},
  number={5},
  pages={1545--1578},
  year={2011},
  doi={10.1111/j.1540-6261.2011.01681.x},
  publisher={Wiley Online Library}
}

@book{morenoRochet2018,
  title={Continuous-Time Models in Corporate Finance, Banking, and Insurance: A User's Guide},
  author={Moreno-Bromberg, Santiago and Rochet, Jean-Charles},
  publisher={Princeton University Press},
  address={Princeton and Oxford},
  year={2018},
  doi={10.1515/9781400889204},
  isbn={9780691176529}
}

@article{diamond,
  title={Bank runs, deposit insurance, and liquidity},
  author={Diamond, Douglas W and Dybvig, Philip H},
  journal={Journal of political economy},
  volume={91},
  number={3},
  pages={401--419},
  year={1983},
  publisher={The University of Chicago Press}
}

@techreport{cetina,
  title={The difficult business of measuring banks’ liquidity: Understanding the liquidity coverage ratio},
  author={Cetina, Jill and Gleason, Katherine},
  institution={Office of Financial Research},
  type={Working Paper},
  number={15-20},
  year={2015}
}

@article{bala,
  title={Bank balance sheet dynamics under a regulatory liquidity-coverage-ratio constraint},
  author={Balasubramanyan, Lakshmi and VanHoose, David D},
  journal={Journal of Macroeconomics},
  volume={37},
  pages={53--67},
  year={2013},
  publisher={Elsevier}
}

@article{doer,
  title={The liquidity coverage ratio a decade on: a stocktake of the literature},
  author={Doerr, Sebastian and Drehmann, Mathias},
  journal={Available at SSRN 5758225},
  year={2025}
}

@article{harrison1,
  title={Instantaneous control of Brownian motion},
  author={Harrison, J Michael and Taksar, Michael I},
  journal={Mathematics of Operations research},
  volume={8},
  number={3},
  pages={439--453},
  year={1983},
  publisher={INFORMS}
}

@article{harrison2,
  title={Impulse control of Brownian motion},
  author={Harrison, J Michael and Sellke, Thomas M and Taylor, Allison J},
  journal={Mathematics of Operations Research},
  volume={8},
  number={3},
  pages={454--466},
  year={1983},
  publisher={INFORMS}
}

@article{orm,
  title={Impulse control of Brownian motion: The constrained average cost case},
  author={Ormeci, Melda and Dai, Jim G and Vate, John Vande},
  journal={Operations Research},
  volume={56},
  number={3},
  pages={618--629},
  year={2008},
  publisher={INFORMS}
}

@article{decamps2,
  title={Free cash flow, issuance costs, and stock prices},
  author={D{\'e}camps, Jean-Paul and Mariotti, Thomas and Rochet, Jean-Charles and Villeneuve, St{\'e}phane},
  journal={The Journal of Finance},
  volume={66},
  number={5},
  pages={1501--1544},
  year={2011},
  publisher={Wiley Online Library}
}

@article{rochet,
  title={Capital requirements and the behaviour of commercial banks},
  author={Rochet, Jean-Charles},
  journal={European economic review},
  volume={36},
  number={5},
  pages={1137--1170},
  year={1992},
  publisher={Elsevier}
}

@misc{acquiredDimon2025,
  author       = {Gilbert, Ben and Rosenthal, David},
  title        = {The {Jamie Dimon} Interview},
  year         = {2025},
  month        = jul,
  howpublished = {Acquired podcast,
                  \url{https://www.acquired.fm/episodes/the-jamie-dimon-interview}},
  note         = {Interview with Jamie Dimon, especially 44:46--46:16;
                  accessed August 14, 2026}
}

@misc{dimon2017,
  author       = {Dimon, Jamie},
  title        = {Letter to Shareholders, 2016 Annual Report},
  year         = {2017},
  month        = apr,
  organization = {JPMorgan Chase \& Co.},
  howpublished = {\url{https://reports.jpmorganchase.com/investor-relations/2016/ar-ceo-letters.htm}},
  note         = {Dated April 4, 2017; accessed August 14, 2026}
}

@misc{dimon2026,
  author       = {Dimon, Jamie},
  title        = {Chairman and {CEO} Letter to Shareholders,
                  2025 Annual Report},
  year         = {2026},
  month        = apr,
  organization = {JPMorgan Chase \& Co.},
  howpublished = {\url{https://www.jpmorganchase.com/ir/annual-report/2025/ar-ceo-letters}},
  note         = {Dated April 6, 2026; accessed August 14, 2026}
}

@article{LokkaZervos2008,
  author  = {L{\o}kka, Arne and Zervos, Mihail},
  title   = {Optimal dividend and issuance of equity policies in the
             presence of proportional costs},
  journal = {Insurance: Mathematics and Economics},
  volume  = {42},
  number  = {3},
  pages   = {954--961},
  year    = {2008},
  doi     = {10.1016/j.insmatheco.2007.10.013}
}
\end{document}